\documentclass[10pt,final,reqno]{amsart}
\usepackage{amsfonts,latexsym,url,amsmath,amssymb,comment,graphicx,enumitem,pgfplots,pgfplotstable}

\usepackage{upgreek}
\excludecomment{commentforus}
\usepackage{subfigure}
\usepackage{float}
\usepackage{url,centernot}
\usepackage{graphicx,color}
\usepackage[color,notref,notcite]{showkeys}
\usepackage{bm}
\usepackage{stmaryrd}
\usepackage{citesort}
\usepackage{textcomp}
\usepackage{graphicx}
\usepackage{epstopdf}
\usepackage{pdfpages}
\usepackage{subfig,epsfig}
\usepackage{textcomp}
\usepackage{color}
\usepackage{tikz}
\usepackage{caption}

\usepackage{cancel}

\usetikzlibrary{arrows}
\usetikzlibrary{shapes,calc}

\usepackage[square,sort&compress,comma,numbers]{natbib}

\DeclareMathAlphabet\mathbfcal{OMS}{cmsy}{b}{n}

\newcommand{\mathbi}[1]{{\boldsymbol #1}}
\def\N{\mathbb{N}}

\def\R{\mathbb{R}}

\def\err{\mathsf{err}}

\def\dist{{\rm dist}}

\def\P{\mathcal{P}}

\def\<{\langle}
\def\>{\rangle}

\def\dsp{\displaystyle} 
\def\div{{\rm div}}

\def\norm#1#2{\Vert#1\Vert_{#2}}

\newcounter{cst}

\def\bt{\begin{theorem}}
\def\et{\end{theorem}}
\def\bl{\begin{lemma}}
\def\el{\end{lemma}}
\def\bc{\begin{corollary}}
\def\ec{\end{corollary}}
\def\bd{\begin{definition}}
\def\ed{\end{definition}}
\def\br{\begin{remark}}
\def\er{\end{remark}}

\newcommand{\disc}{{\mathcal D}}

\def \hessian{\mathcal H}
\def \hb{\hessian^{B}}

\def \hbd{\hessian_\disc^{B}}
\def \wdspace{H^B}
\def\cv{K}
\def \symd{\mathcal S_d}
\def \sym2{\mathcal S_2}
\def\cell{K}

\newcommand{\polyd}{{\mathcal T}}

\newcommand{\mesh}{{\mathcal M}}

\newcommand{\edge}{{\sigma}}

\newcommand{\edges}{{\mathcal F}}              
\newcommand{\edgescv}{{{\edges}_\cv}}  
  
\newcommand{\edgesext}{{{\edges}_{\rm ext}}} 
\newcommand{\edgesint}{{{\edges}_{\rm int}}}

\newcommand{\centers}{\mathcal{P}}
\newcommand{\x}{\mathbi{x}}

\newcommand{\centeredge}{\overline{\mathbi{x}}_\edge}

\newcommand{\bu}{{u}}

\def\eaaD{{\rm E}} 
\def\stab{\mathfrak{S}}

\newcommand{\be}{\begin{equation}}
\newcommand{\ee}{\end{equation}}

\renewcommand{\O}{\Omega}

\def\dr{\partial}

\renewcommand{\d}{{\:\rm d}}

\newcommand{\ba}{\begin{array}{llll}   }
\newcommand{\bac}{\begin{array}{c}}
\newcommand{\bari}{\begin{array}{r}}
\newcommand{\ea}{\end{array}}

\newcommand{\NORM}[1]{{\left\vert\kern-0.25ex\left\vert\kern-0.25ex\left\vert #1 
    \right\vert\kern-0.25ex\right\vert\kern-0.25ex\right\vert}}

\def\discs{{\disc^*}}

\newtheorem{theorem}{Theorem}[section]
\newtheorem{remark}[theorem]{Remark}

\newtheorem{lemma}[theorem]{Lemma} 
\newtheorem{definition}[theorem]{Definition}
\newtheorem{proposition}[theorem]{Proposition}
\newtheorem{corollary}[theorem]{Corollary}
\newtheorem{assumption}[theorem]{Assumption}

\numberwithin{equation}{section}
\def\WS{{\rm WS}}

\def\XXint#1#2#3{{\setbox0=\hbox{$#1{#2#3}{\int}$ }
\vcenter{\hbox{$#2#3$ }}\kern-.6\wd0}}

\usepackage[normalem]{ulem}
\definecolor{violet}{rgb}{0.580,0.,0.827}

\usepackage{color,colordvi}

\newcounter{cexp}
\catcode`\@=11
\def\terml#1{T_{\refstepcounter{cexp}\@bsphack
		\protected@write\@auxout{}%
		{\string\newlabel{#1}{{\thecexp}{\thepage}}}\thecexp}}
\catcode`\@=12

\begin{document}
\title[Improved $L^2$ and $H^1$ error estimates for the HDM]{Improved $L^2$ and $H^1$ error estimates for the hessian discretisation method}
\author{Devika Shylaja}
\address{IITB-Monash Research Academy, Indian Institute of Technology Bombay, Powai, Mumbai 400076, India.
	\texttt{devikas@math.iitb.ac.in}}
\maketitle

\begin{abstract}
The Hessian discretisation method (HDM) for fourth order linear elliptic equations provides a unified convergence analysis framework based on three properties namely coercivity, consistency, and limit-conformity. Some examples that fit in this approach include conforming and nonconforming finite element methods, finite volume methods and methods based on gradient recovery operators. A generic error estimate has been established in $L^2$, $H^1$ and $H^2$-like norms in literature. In this paper, we establish improved $L^2$ and $H^1$ error estimates in the framework of HDM and illustrate it on various schemes. Since an improved $L^2$ estimate is not expected in general for finite volume method (FVM), a modified FVM is designed by changing the quadrature of the source term and a superconvergence result is proved for this modified FVM. In addition to the Adini nonconforming finite element method (ncFEM), in this paper, we show that the Morley ncFEM is an example of HDM. Numerical results that justify the theoretical results are also presented.
\end{abstract}

\medskip

{\scriptsize
	\textbf{Keywords}: fourth order elliptic equations, numerical schemes, error estimates, Hessian discretisation method, Hessian schemes, finite element method, finite volume method, gradient recovery method. 
	
	\smallskip
	
	\textbf{AMS subject classifications}: 65N08, 65N12, 65N15, 65N30.
}

\section{Introduction}
There are many applications where fourth order elliptic partial differential equations appear, for example, thin plate theories of elasticity \cite{ciarlet_plate}, thin beams and the Stokes problem in stream function and vorticity formulation \cite{Lions_NS}. Consider the following fourth order model problem with homogeneous clamped boundary conditions.
\begin{subequations} \label{model_problem}
	\begin{align}
	& {\sum_{i,j,k,l=1}^{d}\partial_{kl}(a_{ijkl}\partial_{ij}\bu) = f  \quad \mbox{ in } \Omega, }\label{problem}\\
	&\qquad \qquad \quad \bu=\frac{\partial \bu}{\partial n}= 0\quad\mbox{ on $\partial\O$}, \label{bc1}
	\end{align}
\end{subequations} 
where $\O \subset \R^d (d\ge1)$ is a bounded domain  with boundary $\partial \Omega$,  $f \in L^2(\O)$ and $n$ is the unit outward normal to $\O$. Furthermore, the coefficients $a_{ijkl}$ are measurable bounded functions which satisfy the conditions $ a_{ijkl}=a_{jikl}=a_{ijlk} =a_{klij}$ for $i,j,k,l=1,\cdots,d .$ 
\medskip

The Hessian discretisation method (HDM) for fourth order linear elliptic equations is a unified convergence analysis framework based on the choice of a set of discrete space and operators called altogether a Hessian discretisation (HD). The idea of the HDM is to construct a scheme by replacing the continuous space, function, gradient, and Hessian in the weak formulation with the discrete elements provided by a HD. The numerical scheme thus obtained is called a Hessian scheme. The concept of HDM is motivated by the Gradient discretisation method (GDM) \cite{gdm} for second order problems. The framework of HDM enables us to develop one study that encompasses several numerical methods such as conforming and nonconforming finite element methods, finite volume methods and methods based on gradient recovery operators. It has been shown in \cite{HDM_linear} that only three properties, namely coercivity, consistency, and limit-conformity, are sufficient to prove the convergence of a HDM.
\medskip

The finite element method (FEM) is one of the most well-known tools for solving fourth-order elliptic problems. Conforming finite element (for e.g., the Argyris triangle, the Bogner--Fox--Schmit rectangle) methods for \eqref{model_problem} requires the approximation space to be a subspace of $H^2_0(\O)$, which results in $C^1$ finite elements that is cumbersome for implementations \cite{dou_percell_scott,ciarlet1978finite,percell_ctfem}. The nonconforming Morley elements which are based on piecewise quadratic polynomials are simpler to use and have fewer degrees of freedom (6 degrees of freedom in a triangle). The Adini element is a well-known nonconforming finite element on rectangular meshes with 12 degrees of freedom in a rectangle. For an analysis of finite element approximation by a mixed method, see \cite{FR78,BR77}. 
\medskip

In \cite{BL_FEM}, a finite element approximation based on gradient recovery (GR) operator for a biharmonic problem using biorthogonal system has been studied, where the approximation properties of the GR operator ensure the optimality of the finite element approach. The GR operator maps an $L^2$ function to a piecewise linear globally continuous $H^1$ function and this enables to define a Hessian matrix starting from $\mathbb{P}_1$ functions, see \cite{BL_MFEM,BL_stab.mixedfem,BL_FEM} for more details. A cell centered finite volume method (FVM) for the approximation of a biharmonic problem has been proposed and analyzed in \cite{biharmonicFV}, first on grids which satisfy an orthogonality condition, and then on general meshes. This scheme consists of approximation by piecewise constant functions and hence it is easy to implement and computationally cheap. 
\medskip

A generic error estimate has been established for the HDM applied to \eqref{model_problem} in \cite{HDM_linear}. This estimate only gives linear order of convergence in $L^2$, $H^1$ and $H^2$ norms for low-order conforming FEMs, Adini nonconforming FEM and methods based on GR operators, provided $\bu \in H^4(\O)\cap H^2_0(\O)$. Also, the error estimate provides an $\mathcal{O}(h^{1/4}|\ln(h)|)$ (in $d=2$) or $\mathcal O(h^{3/13})$ (in $d=3$) convergence rate for the FVM in the HDM framework, where $h$ denotes the mesh parameter. However, an $\mathcal{O}(h^2)$ superconvergence rate in $L^2$ norm has been numerically observed in \cite{HDM_linear} on two dimensional triangular and square meshes. Note that the FVM only works for the biharmonic problem with the approximation of the Laplacian of the functions while the other methods work for more generic fourth-order problems in the HDM setting. 
\medskip

The goal of this paper is to obtain an improved error estimate in $L^2$ and $H^{1}$-like norms compared to the estimate in the energy norm for the HDM applied to \eqref{model_problem}. The Aubin--Nitsche duality arguments apply to establish $L^2$ and $H^{1}$ estimates in the abstract framework which involve an interpolant of the solution to \eqref{model_problem} in the weak sense. However, for the $H^1$ error estimate, this is not straightforward. Under the assumption that there exists a companion operator that lifts the discrete space to the continuous space with certain property, an improved $H^1$ error estimate is proved in the abstract setting. These estimates are then illustrated for some schemes contained in the HDM framework. Since such an improved $L^2$ estimate is not true in general for FVM even in the case of second order problems (\cite{jd_nn} and references therein), a modified FVM is designed in which only the right hand side in the Hessian scheme is modified and a superconvergence result is proved for this modified method. In addition, it is also established that the Morley nonconforming finite element method (ncFEM) is an example of HDM. Numerical experiments are performed to validate the theoretical estimates for the GR method and modified FVM.
\medskip

The rest of this article is organised as follows. Section \ref{sec.wf} deals with the weak formulation of the model problem \eqref{model_problem}. Section \ref{sec.hdm} briefly describes the Hessian discretisation method and states the basic error estimates. Some examples of HDM are presented in Subsection \ref{sec.HDMeg}. The improved $L^2$ and $H^1$ error estimates for the HDM are stated in Section \ref{sec.results} and a modified FVM is designed. These estimates are then applied to several schemes. This section also states the convergence of Morley ncFEM in the HDM framework. Numerical results for the gradient recovery method and the modified FVM are presented in Section \ref{sec.eg}. Section \ref{sec.proof} deals with the proof of the main results. Section \ref{appendix} is an Appendix, that gathers various results: some technical results and the proof of the application of improved error estimates to various schemes stated in Section \ref{sec.results}.
\medskip

\textbf{Notations}. Let  $d$ be the dimension and $\symd(\R)$ be the space of symmetric matrices. A fourth order symmetric tensor $P$ is interpreted as a linear map from $\symd(\R)$ to $\symd(\R)$ and let $p_{ijkl}$ denote the indices of the fourth order tensor $P$ in the canonical basis of $\symd(\R)$. For simplicity, we follow the Einstein summation convention unless otherwise stated. The scalar product on $\symd(\R)$ is defined by $\xi:\phi=\xi_{ij}\phi_{ij}$. For a function $\xi: \O \rightarrow \symd(\R)$, denoting the Hessian matrix by $\hessian$ we set $\hessian : \xi = \partial _{ij}\xi_{ij}$. The transpose $P^{\tau}$ of $P$ is given by $P^{\tau} = (p_{klij})$, if $P = (p_{ijkl})$. Note that ($P\phi)_{ij}=p_{ijkl}\phi_{kl}$ and $P^{\tau}\xi : \phi = \xi : P \phi$. The tensor product $a\otimes b$ of two vectors $a,b\in\R^d$ is the 2-tensor with coefficients $a_ib_j$. The Lebesgue measure of a measurable set $E\subset \R^d$ is denoted by $|E|$. The norm in $L^2(\O)$, $L^2(\O)^d$ for vector-valued functions, and $L^2(\O;\R^{d\times d})$ for matrix-valued functions, is denoted by $\|{\cdot}\|$. We denote by $(\cdot,\cdot)$ the $L^2$ inner product or duality pairing between $H^{-1}(\O)$ and $H^1_0(\O),$ this could be understood from the context. 
\subsection{Weak formulation}\label{sec.wf}
The weak formulation corresponding to \eqref{model_problem} reads:
\be
\mbox{Find  $\bu \in V:=H^2_0(\O)$ such that $\forall v \in V$,}\quad
\int_\O A\hessian \bu:\hessian v \d\x=\int_\O fv\d\x, \label{weak.general}
\ee
where $A$ is the fourth order tensor with indices $a_{ijkl}$ and $ \x=(x_1,x_2,...,x_d) \in \O$. Assume the existence of a fourth order tensor $B$ such that for all $\xi, \phi \in \symd(\R)$, $A\xi:\phi =B\xi:B\phi$. Since $B^{\tau}\xi : \phi = \xi : B \phi$, we obtain $A = B^{\tau}B$.
\smallskip

The weak formulation \eqref{weak.general} corresponding to \eqref{model_problem} can be rewritten as
\be\label{weak}
\mbox{Find  $\bu \in V$ such that $\forall v \in V$,}
\quad a(\bu,v)=\int_\O fv\d\x,
\ee
where 
\be\label{weak.a}
a(\bu,v)= \int_\O \hb \bu:\hb v \d\x \mbox{  and  } \hb v=B\hessian v.
\ee 
We assume in the following that $B$ is constant over $\O$, and that the following coercivity property holds:
\be \label{coer:B}
\exists \varrho>0\mbox{ such that } \norm{\hb v}{}\ge \varrho\norm{v}{H^2(\Omega)}\;\forall v\in H^2_0(\Omega).
\ee
Hence, the weak formulation \eqref{weak} has a unique solution by the Lax--Milgram lemma. Note that we do not necessarily discretise the full Hessian matrix and this is the purpose of the introduction of the tensors $A$ and $B$. Even for the biharmonic problem, which could be dealt with using just $B$ the identity tensor ($B\xi=\xi$), there is an interest in introducing other possible tensors that lead to the same model. Precisely because the weak formulation with $B\xi=\xi$ requires to use and discretise the entire Hessian matrix, whereas other choices of $B$, such as $B\xi=\frac{{\rm tr}(\xi)}{\sqrt{d}}{\rm Id}$ (where ${\rm tr}(\xi)$ is the trace of $\xi$ and ${\rm Id}$ is the identity matrix), lead to a weak formulation that only involves the Laplacian, and thus whose numerical approximation only requires to approximate this particular operator (not each and every second order derivative and with the full Hessian). In this paper, the FVM is built on an approximation of the Laplacian of the functions whereas the FEMs work with a generic $A$ that is independent of the model. An overview of the choice of $B$ for biharmonic and plate problems can be found in \cite{HDM_linear}.


  \section{The Hessian discretisation method}\label{sec.hdm}
The HDM \cite{HDM_linear} for fourth order linear elliptic partial differential equations is briefly presented in this section. The HDM consists in writing a scheme, known as a Hessian scheme (HS), by replacing the space and the continuous operators in the weak formulation \eqref{weak} with discrete components. These discrete components are provided by a Hessian discretisation (HD). 
 \begin{definition}[$B\textendash$Hessian discretisation]\label{HD}
 	A $B\textendash$Hessian discretisation for homogeneous clamped boundary conditions is a quadruplet $\disc=(X_{\disc,0},\Pi_\disc,\nabla_\disc,\hbd)$ such that
 	\begin{itemize}
 		\item $X_{\disc,0}$ is a finite-dimensional space encoding the unknowns of the method,
 		\item $\Pi_\disc:X_{\disc,0}  \rightarrow L^2(\O)$ is a linear mapping that reconstructs a function from the unknowns,
 		\item $\nabla_\disc:X_{\disc,0}  \rightarrow L^2(\O)^d$ is a linear mapping that reconstructs a gradient from the unknowns,
 		\item $\hbd:X_{\disc,0}  \rightarrow L^2(\O;\R^{d\times d})$ is a linear mapping that reconstructs a discrete version of $\hb(=B\hessian)$ from the unknowns. It must be chosen such that $\norm{\cdot}{\disc}:=\norm{\hbd \cdot}{}$ is a norm on  $X_{\disc,0}.$
 	\end{itemize}
 \end{definition}
 
Let $\disc=(X_{\disc,0},\Pi_\disc,\nabla_\disc,\hbd)$ be a $B\textendash$Hessian discretisation. Then the related HS for \eqref{weak} is given by
 \be\label{base.HS}
 \begin{aligned}
 	&\mbox{Find $u_\disc\in X_{\disc,0}$ such that for any $v_\disc\in X_{\disc,0}$,}\\
 	&a_\disc(u_\disc,v_\disc)=\int_\O f\Pi_\disc v_\disc\d\x,
 \end{aligned}
 \ee
 where $a_\disc(u_\disc,v_\disc)=\int_\O \hbd u_\disc:\hbd v_\disc\d\x.$

 \subsection{Basic error estimates}
 Given a Hessian discretisation $\disc$, the accuracy of a Hessian scheme is measured by three quantities.
 
 The first one is a constant, a measure of coercivity, which controls the norm of $\Pi_\disc$ and $\nabla_\disc$.
 \be\label{def.CD}
 C_\disc^B = \max_{w_\disc\in X_{\disc,0}\backslash\{0\}} \left(\frac{\norm{\Pi_\disc w_\disc}{}}{{\norm{\hbd w_\disc}{}}},
 \frac{\norm{\nabla_\disc w_\disc}{}}{{\norm{\hbd w_\disc}{}}}\right).
 \ee

 The second measure involves an estimate of the interpolation error in the finite element framework, called the consistency in the framework of the HDM. 
 \be\label{def.SD}
 \begin{aligned}
 	\forall  \varphi\in H^2_0(\O)\,,\;
 	S_\disc^B(\varphi)=\min_{w_\disc\in X_{\disc,0}}\Big(&\norm{\Pi_\disc w_\disc-\varphi}{}
 	+\norm{\nabla_\disc w_\disc-\nabla\varphi}{}\\	
 	&+\norm{\hbd w_\disc-\hb \varphi}{}\Big).
 \end{aligned}
 \ee

 Finally, the third quantity measures the error in the discrete integration by parts known as the limit--conformity and is defined by
 \be\label{def.WD}
 \begin{aligned}
 	&\forall \, \xi \in \wdspace(\O)\,,\,W_\disc^B(\xi)=\max_{w_\disc\in X_{\disc,0}\backslash\{0\}}
 	\frac{\Big|\mathcal{W}_\disc^B(\xi, w_\disc)  \Big|}{\norm{\hbd w_\disc}{}},
 \end{aligned}
 \ee
 where $\wdspace(\O)=\{\xi\in L^2(\O)^{d \times d}\,;\,\hessian:B^{\tau}B\xi \in L^2(\O) \}$ and
 \be\label{def:W.nr}
 \mathcal{W}_\disc^B(\xi, w_\disc) = \int_\O \Big((\hessian:B^{\tau}B\xi)\Pi_\disc w_\disc - B\xi:\hbd w_\disc \Big)\d\x. 
 \ee

The notation $X\lesssim Y$ means that $X\le CY$ for some $C$ depending only on $\O$ and an upper bound of $C_\disc^B$.

 \begin{theorem}[Error estimate for Hessian schemes]\cite[Theorem 3.6]{HDM_linear}\label{th:error.est.PDE} Let $\disc$ be a $B\textendash$Hessian discretisation in the
 	sense of Definition \ref{HD}, $\bu$ be the solution
 	to \eqref{weak} and $u_\disc$  be the solution to \eqref{base.HS}. Then
 	\begin{align}\label{err.est.PDE}
 	\norm{\Pi_\disc u_\disc-\bu}{}+\norm{\nabla_\disc u_\disc-\nabla\bu}{}+	\norm{\hbd u_\disc-\hb \bu}{}
 	\lesssim \WS_\disc^B(\bu),
 	\end{align}
 	where 
 	\be \label{def.ws}
 	\WS_\disc^B(\bu)= W_\disc^B(\hessian \bu)+S_\disc^B(\bu).
 	\ee
 
 \end{theorem} 
\begin{remark}[Convergence of the HS]
	Along a sequence $(\disc_m)_{m \in \N}$ of $B\textendash$Hessian discretisations, it is expected that $C^B_{\disc_m}$ remains bounded, $S^B_{\disc_m}(\varphi) \rightarrow 0$ for all $\varphi \in H^2_0(\O)$ and $W^B_{\disc_m}(\xi) \rightarrow 0$ for all $\xi \in \wdspace(\O)$ as $m \rightarrow \infty$ (see for example Theorem \ref{th.Morley}). Then Theorem \ref{th:error.est.PDE} gives the convergence of the HS along sequences of such HDs. 
\end{remark}
 \subsection{Examples of HD}\label{sec.HDMeg}
A few examples of $B\textendash$HD are presented in this section. We refer to \cite{HDM_linear} for a detailed analysis of these methods. In addition, it is established that the Morley ncFEM is an example of HDM. Let us first set some notations related to meshes.
 \begin{definition}[Polytopal mesh {\cite[Definition 7.2]{gdm}}]\label{def:polymesh}~
 	Let $\Omega$ be a bounded polytopal open subset of $\R^d$ ($d\ge 1$). A polytopal mesh of $\O$ is $\polyd = (\mesh,\edges,\centers)$, where:
 	\begin{enumerate}
 		\item $\mesh$ is a finite family of non empty connected polytopal open disjoint subsets of $\O$ (the cells) such that $\overline{\O}= \dsp{\cup_{\cell \in \mesh} \overline{\cell}}$. For any $\cell\in\mesh$, $|\cell|>0$ is the measure of $\cell$, $h_\cell$ denotes the diameter of $\cell$, $\overline{\x}_\cell$ is the center of mass of $K$, and $n_K$ is the outer unit normal to $K$.
 		
 		\item $\edges$ is a finite family of disjoint subsets of $\overline{\O}$ (the edges of the mesh in 2D, the faces in 3D), such that any $\edge\in\edges$ is a non empty open subset of a hyperplane of $\R^d$ and $\edge\subset \overline{\O}$. Assume that for all $\cell \in \mesh$ there exists  a subset $\edgescv$ of $\edges$ such that the boundary of $\cell$ is ${\bigcup_{\edge \in \edgescv}} \overline{\edge}$. We then set $\mesh_\edge = \{\cell\in\mesh\,;\,\edge\in\edgescv\}$ and assume that, for all $\edge\in\edges$, $\mesh_\edge$ has exactly one element and $\edge\subset\partial\O$, or $\mesh_\edge$ has two elements and $\edge\subset\O$. Let $\edgesint$ be the set of all interior faces, i.e. $\edge\in\edges$ such that $\edge\subset \O$, and $\edgesext$ the set of boundary faces, i.e. $\edge\in\edges$ such that $\edge\subset \dr\O$. The $(d-1)$-dimensional measure of $\edge\in\edges$ is $|\edge|$, and its centre of mass is $\centeredge$.
 		
 		\item $\centers = (\x_\cell)_{\cell \in \mesh}$ is a family of points of $\O$ indexed by $\mesh$ and such that, for all  $\cell\in\mesh$,  $\x_\cell\in \cell$.
 		Assume that any cell $\cell\in\mesh$ is strictly $\x_\cell$-star-shaped, meaning that 
 		if $\x\in\overline{\cell}$ then the line segment $[\x_\cell,\x)$ is included in $\cell$. 
 		
 	\end{enumerate}
 	The diameter of such a polytopal mesh is $h=\max_{K\in\mesh}h_K$. The set of internal vertices of $\mesh$ (resp. vertices on the boundary) is denoted by $\mathcal{V}_{\rm int}$ (resp. $\mathcal{V}_{\rm ext}$). 
 \end{definition}
 
We assume that $\mesh=\mesh_h$ satisfies minimal regularity assumptions. That is, if $\rho_\cell=\max\{r>0\,:\,B(\overline{\x}_\cell,r)\subset \cell\}$, then there exists $\eta>0$, independent of $h$, such that
 $\forall \cell\in\mesh\,,\; \frac{h_K}{\rho_\cell} \le \eta.$

\subsubsection{Conforming finite elements}\label{sec.conforming} The $B\textendash$HD $\disc=(X_{\disc,0},\Pi_\disc,\nabla_\disc,\hbd)$ for conforming FEM is defined by: $X_{\disc,0}$ is a finite dimensional subspace of $H^2_0(\O)$ and, for $v_\disc\in X_{\disc,0}$, $\Pi_\disc v_\disc=v_\disc$, $\nabla_\disc v_\disc=\nabla v_\disc$ and $\hbd v_\disc=\hb v_\disc$. Examples of conforming finite elements include the Argyris and Bogner--Fox--Schmit (BFS) finite elements, see \cite{ciarlet1978finite} for details.

\subsubsection{Non-conforming finite elements}\label{sec.ncFEM.eg}\textbf{}\\
$\bullet$ \textsc{the Adini rectangle \cite{ciarlet1978finite}: }Assume that $\O\subset\R^2$ can be covered by a mesh $\mesh$ made up of rectangles. Figure \ref{ncfem.fig} (left) represents an Adini rectangle $\cell \in \mesh$ with vertices $a_1,\,a_2,\,a_3$ and $a_4$ respectively. Each $v_\disc\in X_{\disc,0}$  is a vector of three values at each vertex of the mesh (with zero values at boundary vertices), corresponding to function and gradient values, $\Pi_\disc v_\disc$ is the function such that the values of $(\Pi_\disc v_\disc)_{|K}\in \mathbb{P}_3 \oplus \{x_1x^3_2\} \oplus\{x^3_1x_2\}$ and its gradients at the vertices are dictated by $v_\disc$, $\nabla_\disc v_\disc=\nabla(\Pi_\disc v_\disc)$ and $\hbd v_\disc= \hb_{\mesh}(\Pi_\disc v_\disc)$ is the broken $\hb$ of $\Pi_\disc v_\disc$. 
\begin{figure}
	\begin{center}
		\begin{minipage}[b]{0.4\linewidth}
			{\includegraphics[width=5.cm]{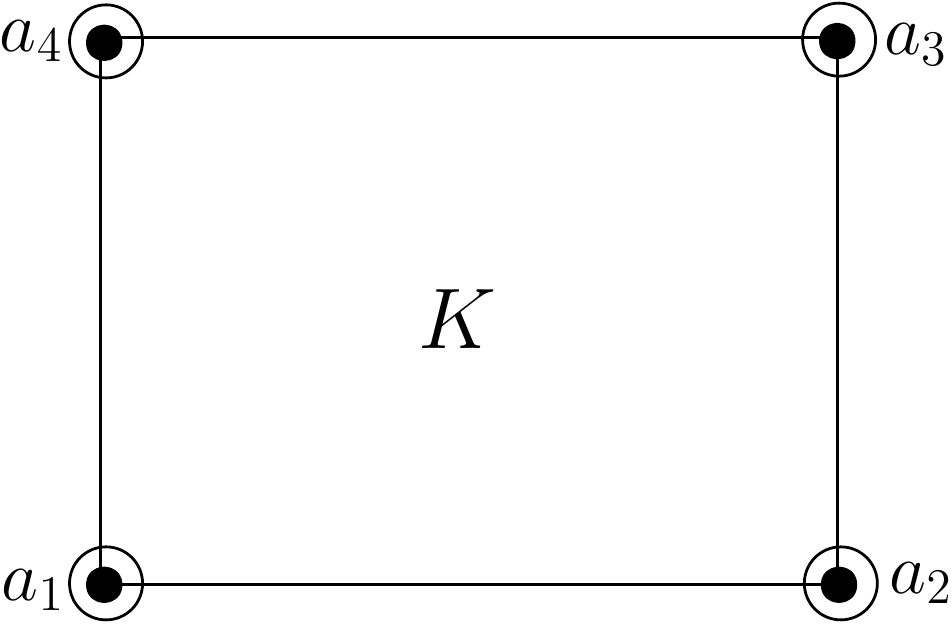}}
		\end{minipage}
		\qquad \quad
		\begin{minipage}[b]{0.35\linewidth}
			{\includegraphics[width=4.3cm]{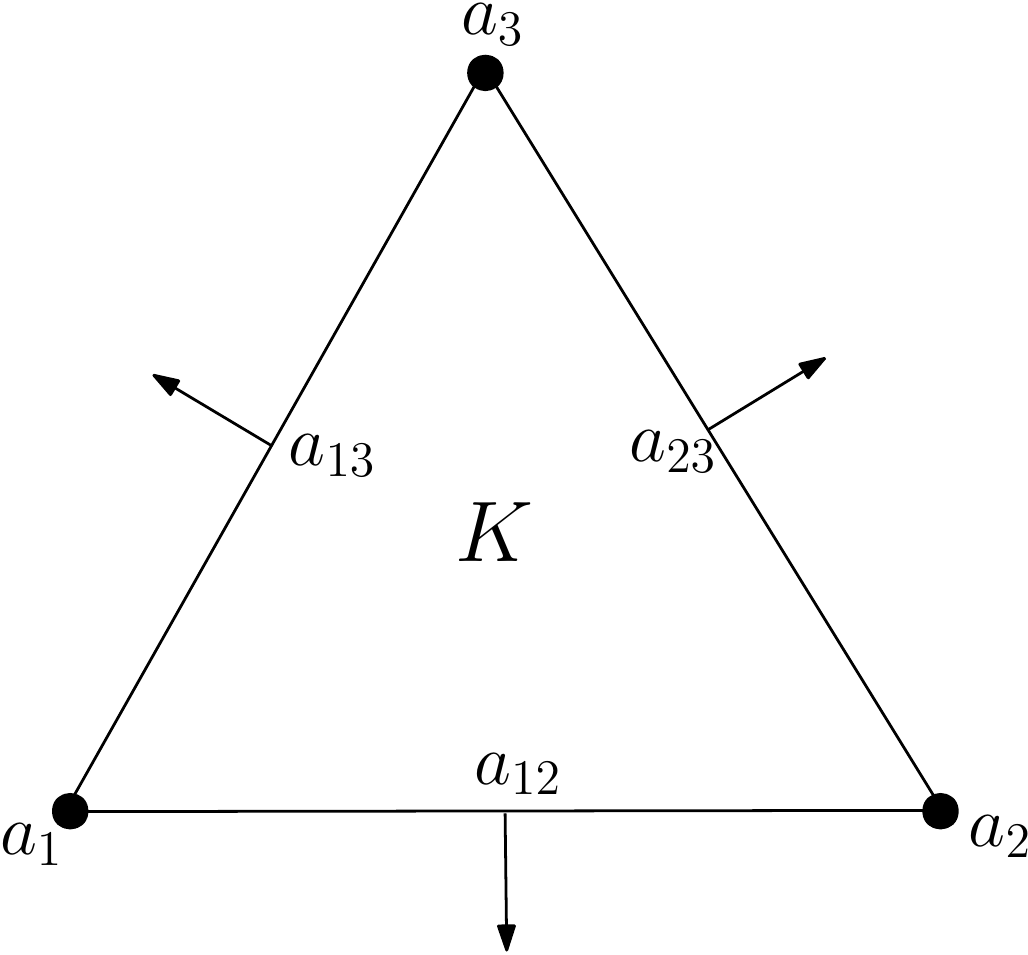}}
		\end{minipage}
		\caption{Adini element (left) and Morley element (right) }\label{ncfem.fig}
	\end{center}
\end{figure}
 \smallskip
 
 $\bullet$ \textsc{the Morley element \cite{ciarlet1978finite}: }We recast here the classical nonconforming FEM, the Morley ncFEM, in the Hessian discretisation method with $d=2$. Let $\mesh$ be a regular conforming triangulation of $\overline{\O}$ into closed triangles (see Figure \ref{ncfem.fig}, right). The Morley finite element is a triplet $(\cell,\mathbb{P}_\cell, \Sigma_\cell)$ where $\cell$ is a triangle, $\mathbb{P}_\cell$ = $\mathbb{P}_2(\cell)$, space of all polynomials of degree $\le$ 2 in two variables defined on $\cell$ (dim $\mathbb{P}_\cell=6$) and $\Sigma_\cell$ denote the degrees of freedom  consist of the values at the vertices of the mesh and normal derivatives at the midpoints of the edges opposite to these vertices.
 
Let $\mathbb{P}_2(\mesh)$ denote the space of all piecewise polynomials of degree atmost equal to 2 defined on $\mesh$. Then the nonconforming Morley element space associated with $\mesh$ is defined by
 \begin{align*}
 V_h & =: \bigg\{\phi \in \mathbb{P}_2(\mesh) |\phi \mbox{ is continuous at } \mathcal{V}_{\rm int}\mbox{ and vanishes at }\mathcal{V}_{\rm ext},\, \\ &\qquad\forall \edge \in \edgesint, \, \int_{\edge}^{}\bigg\llbracket\frac{\partial \phi}{\partial n}\bigg\rrbracket \d s=0;\,\forall \edge \in \edgesext, \, \int_{\edge}^{}\frac{\partial \phi}{\partial n}\d s=0\bigg\},
 \end{align*}
 where $\llbracket \phi \rrbracket$ denote the jump of the function $\phi$ along the edges.

 \begin{definition}[Hessian discretisation for the Morley element]\label{def.Morley}
 	Each $v_\disc\in X_{\disc,0}$  is a vector of degrees of freedom at the vertices of the mesh (with zero values at boundary vertices) and at the midpoint of the edges opposite to these vertices (with zero values at midpoint of the boundary edges). $\Pi_\disc v_\disc$ is the function such that $(\Pi_\disc v_\disc)_{|K}\in \mathbb{P}_\cell$ (resp. its normal derivatives) takes the values at the vertices (resp. at the edge midpoints) dictated by $v_\disc$, $\nabla_\disc v_\disc= \nabla_{\mesh}(\Pi_\disc v_\disc)$ is the broken gradient of $\Pi_\disc v_\disc$ and $\hbd v_\disc= \hb_{\mesh}(\Pi_\disc v_\disc)$ is the broken $\hb$ of $\Pi_\disc v_\disc$. 
 \end{definition}
 
 \subsubsection{Method based on Gradient Recovery Operators}\label{sec.gr}
In this method, the finite element space $V_h$ consists of piecewise linear polynomials, which are continuous over $\O$ and have a zero value on $\partial \O.$ Let $u_h \in V_h$ and let $Q_h: L^2(\O) \rightarrow V_h$ be a gradient recovery projection operator (see, e.g., \cite[Section 4.2]{HDM_linear} for a GR operator based on biorthogonal systems). This gives $Q_h \nabla u_h \in \mathbb{P}_1$, which is differentiable and hence a sort of second derivative of $u_h$ is expressed in terms of $\nabla Q_h\nabla u_h$. In order to ensure the coercivity property of this reconstructed Hessian, we consider a stabilisation function $\stab_h\in L^\infty(\O)^d$ with specific design properties \cite{HDM_linear}. Then the $B$--Hessian discretisation based on a triplet $(V_h,Q_h,\stab_h)$ is defined by: $X_{\disc,0}=V_h$ and, for $u_\disc \in X_{\disc,0}$, $\Pi_\disc u_\disc=u_\disc,\,\nabla_\disc u_\disc = Q_h\nabla u_\disc$ and $\hbd u_\disc=B\left[\nabla (Q_h \nabla u_\disc)+\stab_h\otimes (Q_h\nabla u_\disc-\nabla u_\disc)\right].$

 \subsubsection{Finite volume method based on $\Delta$-adapted discretisations}\label{FVM} Consider the finite volume scheme from \cite{biharmonicFV} for the biharmonic problem on $\Delta$-adapted meshes (see Figure \ref{fig:diagram}). For all  $\edge \in \edgesint$ with $\mesh_{\edge}= \{\cell, L\}$, the straight line $(\x_\cell, \x_L)$ intersects and is orthogonal to $\edge$, and for all $\edge \in \edgesext$ with $\mesh_{\edge}= \{\cell\}$, the line orthogonal to $\edge$ going through $\x_\cell$ intersects $\edge$. Since $\hb=\Delta$ in this method, one possible choice of $B$ is therefore to set $B\xi=\frac{\rm{tr}(\xi)}{\sqrt{d}}{\rm Id}$ for $\xi \in \symd(\R)$ where ${\rm Id}$ is the identity matrix. This method requires only one unknown per cell. 
 
\begin{figure}[H]
\centering
\includegraphics[width=0.7\linewidth]{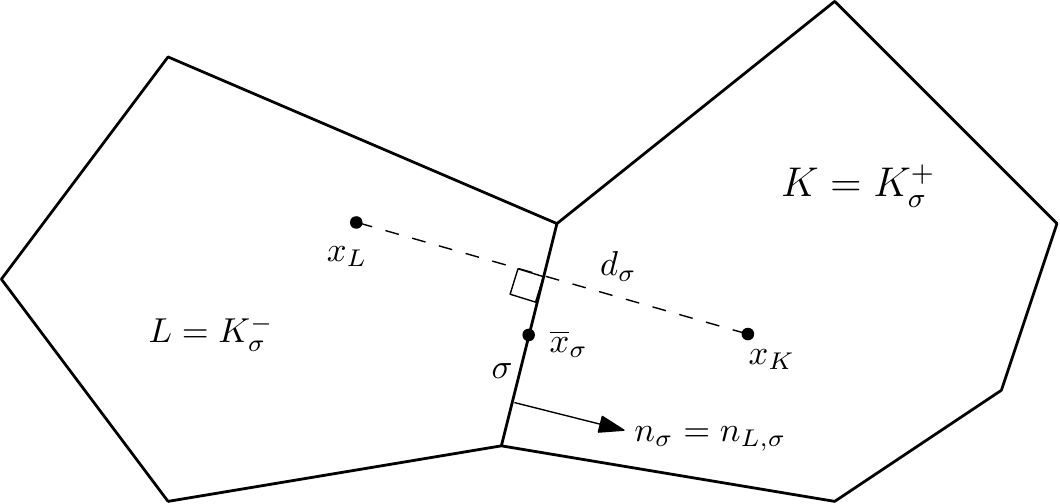}
\caption{Notations for $\Delta$-adapted discretisation}
\label{fig:diagram}
\end{figure}

 $X_{\disc,0}$ is the space of all real families $v_\disc=(v_\cell)_{\cell \in \mesh}$ such that $v_\cell=0$ if $\cell$ touches $\partial \O$. The operator $\Pi_\disc$ reconstructs a piecewise constant function given by: for any cell $K$, $\Pi_\disc v_\disc = v_K$ on $K$. For $\cell \in \mesh$ and $\edge \in \edgescv$, let $n_{\cell,\edge}$ be the unit vector normal to $\edge$ outward to $\cell$. For all $\edge \in \edges$, we choose an orientation (that is, a cell $K$ such that $\edge\in\edges_K$) and set $n_\edge=n_{\cell,\edge}$. For each $\edge \in \edgesint$, denote by $\cell^-_{\edge}$ and $\cell^+_{\edge}$ the two adjacent control volumes such that the unit normal vector $n_{\edge}$ is oriented from $\cell^-_{\edge}$ to $\cell^+_{\edge}$. For all $\edge \in \edgesext$, denote the control volume $\cell \in \mesh$ such that $\edge \in \edgescv$ by $\cell_{\edge}$ and define $n_{\edge}$ by $n_{\cell,\edge}$. Let
 \begin{equation*}
 d_\edge=\left\{
 \begin{array}{ll}
 \dist(\x_{\cell_\edge^-},\edge) + \dist(\x_{\cell_\edge^+},\edge) & \forall \edge \in \edgesint \\
 \dist(\x_{\cell_\edge},\edge)& \forall \edge \in \edgesext
 \end{array}\right.
 \end{equation*}
  where $\dist(\x_{\cell},\edge)$ denotes the orthogonal distance between $\x_\cell$ and $\edge$. The discrete gradient $\nabla_\disc$ and the Laplace operator $\Delta_\disc$  are defined by their constant values on the cells. 
 \begin{equation*}
 \nabla_\cell v_\disc=\frac{1}{|\cell|}\sum_{\edge \in \edges_\cell}^{}\frac{|\edge|(\delta_{\cell,\edge} v_\disc)(\centeredge-\x_\cell)}{d_\edge},  \quad  \Delta_\cell v_\disc=\frac{1}{|\cell|}\sum_{\edge \in \edges_\cell}^{}\frac{|\edge|\delta_{\cell,\edge} v_\disc}{d_\edge}, \label{biharmonic_nabla}
 \end{equation*}
and set $\hbd v_\disc = \frac{\Delta_\disc v_\disc}{\sqrt{d}}{\rm Id}$, where \begin{equation*}
 \delta_{\cell,\edge} v_\disc=\left\{\begin{array}{ll}
 v_L-v_\cell &\forall \,\edge \in  \edgescv \cap \edgesint\,,\; \mesh_{\edge}= \{\cell, L\} \\
 0 & \forall \, \edge \in \edgescv \cap \edgesext.
 \end{array}\right.
 \label{jump}
 \end{equation*}

 \begin{remark}[Rates of convergence \cite{HDM_linear}]\label{rates.PDE}Under regularity assumption $\bu \in H^4(\O)\cap H^2_0(\O)$, for low--order conforming FEMs, Adini ncFEM and gradient recovery methods based on meshes with mesh parameter ``$h$'', $\mathcal O(h)$ estimates can be obtained for $W_\disc^B(\hessian \bu)$ and $S_\disc^B(\bu)$. Theorem \ref{th:error.est.PDE} then gives a linear rate of convergence
 	for these methods. For FVM based on $\Delta$-adapted discretisations, Theorem \ref{th:error.est.PDE} provides an $\mathcal{O}(h^{1/4}|\ln(h)|)$ (in $d=2$) or $\mathcal O(h^{3/13})$ (in $d=3$) error estimate for the Hessian scheme based on the Hessian discretisation. In addition to these results from \cite{HDM_linear}, in this paper, we show that the HDM framework enables us to recover a linear rate of convergence for Morley ncFEM (see Theorem \ref{th.Morley}). 
 	\end{remark}

\section{Main results}\label{sec.results}
The improved $L^2$ and $H^1$ error estimates for HDM are stated in this section. Also, an estimate on the accuracy measures $C_\disc^B$, $S_\disc^B$ and $W_\disc^B$ associated with an HD $\disc$ using Morley ncFEM is stated at the end of this section. The proofs of the results are presented in Section \ref{sec.proof}. The improved error estimates are then applied to the methods listed in Section \ref{sec.hdm}, that is, FEMs, method based on GR operators and slightly modified FVM (see Definition \ref{def.modifiedFVM}). The modified FVM has the same matrix as the original FVM, since only the quadrature of the source term is modified, but enjoys a super-convergence result while the standard FVM fails to super-converge. 

\subsection{Improved $L^2$ error estimate}\label{sec.L2}
For establishing the lower order $L^2$ estimates, consider the adjoint problem corresponding to \eqref{weak}, and its Hessian scheme approximation.

\medskip
The weak formulation for the dual problem with source term $g \in L^2(\O)$ seeks $\varphi_g \in V$ such that
\begin{equation}\label{weak_adjoint}
a(w,\varphi_g)=(g,w)\; \mbox{for all $w\in V$.}
\end{equation}
The Hessian scheme corresponding to  \eqref{weak_adjoint} seeks  $\varphi_{g,\disc} \in X_{\disc,0}$ such that
\be\label{adj.GS}
\begin{aligned}
	&a_{\disc}(w_\disc, \varphi_{g,\disc})=(g, \Pi_\disc w_\disc)\;\mbox{for all $w_\disc\in X_{\disc,0}$.}
\end{aligned}
\ee

\begin{theorem}[Improved $L^2$ error estimate for Hessian schemes] 
	\label{th-l2-super}  \qquad 
	
	\noindent Let $\bu$ be the solution to \eqref{weak}.
	Let $\disc$ be a $B-$Hessian discretisation in the sense of Definition \ref{HD},
	and let $u_\disc$ be the solution to the Hessian scheme \eqref{base.HS}.
	Define
	\[
	g=\frac{\bu-\Pi_\disc u_\disc}{\norm{\bu-\Pi_\disc u_\disc}{}}\in L^2(\O)
	\]
	and let $\varphi_g$ be the solution to \eqref{weak_adjoint}. Choose $\P_\disc \bu,\P_\disc \varphi_g\in X_{\disc,0}$, where $\P_\disc$ is a mapping from $H^2_0(\O)$ to $X_{\disc,0}$.
	Then 
	\be\nonumber 
	\begin{aligned}
		\Vert&\Pi_\disc u_\disc -\bu\Vert_{}
		\lesssim{} \big(\norm{\hbd \P_\disc \bu-\hb \bu}{}+\WS_\disc^B(\bu)\big)
		\big(\norm{\hbd \P_\disc \varphi_g-\hb \varphi_g}{}+\WS_\disc^B(\varphi_g)\big)\\
		&+\norm{\Pi_\disc \P_\disc \bu-\bu}{}
		+\norm{f}{}\norm{\Pi_\disc \P_\disc \varphi_g-\varphi_g}{}+\left|\mathcal{W}_\disc^B(\hessian \bu,\P_\disc\varphi_g)\right|+
		\left|\mathcal{W}_\disc^B(\hessian\varphi_g,\P_\disc \bu)\right|,
	\end{aligned}
	\ee
	where $\WS_\disc^B$ is defined by \eqref{def.ws},
	and $\mathcal{W}_\disc^B$ is defined by \eqref{def:W.nr}.
\end{theorem}
\begin{remark}[Dominating terms]\label{remark.improvedl2.dominate}
Following Remark \ref{rates.PDE}, for FEMs and methods based on GR operators, it is expected that $\WS_D^B(\bu)=\mathcal{O}(h)$ if $\bu \in H^4(\O)\cap H^2_0(\O)$. Hence, for a given HS, Theorem \ref{th-l2-super} provides an improved result if we can find a mapping $\P_\disc$ (usually an interpolant) such that $\norm{\hbd \P_\disc \phi-\hb \phi}{}=\mathcal{O}(h),\, \norm{\Pi_\disc \P_\disc \phi-\phi}{}=\mathcal{O}(h^2),\,\mathcal{W}_\disc^B( \xi,\P_\disc\phi)=\mathcal{O}(h^2)$ for all $\phi \in H^4(\O)\cap H^2_0(\O)$ and all $\xi \in H^2(\O)^{d \times d}.$
\end{remark}
The proof of Theorem \ref{th-l2-super} is presented in Section \ref{sec.proofL2}. We now turn to the application of the above theorem to various schemes described in Section \ref{sec.HDMeg}. The proof of Propositions \ref{ncfem:supercv} and \ref{prop.FVM} are given in Section \ref{appendix}, Appendix. Proposition \ref{ncfem:supercv} justifies the rates numerically observed for the method based on GR operator in \cite{HDM_linear}.

\begin{proposition}\label{ncfem:supercv}Let $\bu\in H^4(\O)\cap H^2_0(\O)$ be the solution to \eqref{weak} and $u_\disc$ be the solution to the Hessian scheme \eqref{base.HS}. Then, for low-order conforming FEMs, Adini and Morley ncFEMs, and gradient recovery methods, there exists a constant $C>0$, not depending on $h$, such that 
		$$\norm{\Pi_\disc u_\disc-\bu}{}\le Ch^2.$$
\end{proposition}
\noindent Since the super-convergence is not known in general for two point flux approximation (TPFA) for second order problems, it is expected that the same issue occurs for the FVM mentioned in Section \ref{FVM}. In order to obtain an improved result, ideas developed in \cite[Section 4]{jd_nn} for GDM is appropriately modified for the HDM. For that, set
\begin{equation}
v_\edge=\left\{\begin{array}{ll}
\frac{\dist(\x_K,\edge)v_L+\dist(\x_L,\edge)v_\cell}{d_\edge} &\forall \,\edge \in  \edgesint\,,\; \mesh_{\edge}= \{\cell, L\} \\
0 & \forall \, \edge \in \edgesext.
\end{array}\right.
\label{vsigma}
\end{equation}
We now define a slightly modified HDM for FVM based on $\Delta$-adapted discretisations. 
\begin{definition}[Modified FVM $B-$HD]\label{def.modifiedFVM}
	Let $\disc=(X_{\disc,0},\Pi_\disc,\nabla_\disc,\hbd)$ be a FVM $B-$Hessian discretisation given in Section \ref{FVM}. The modified FVM $B-$Hessian discretisation is $\disc^*=(X_{\disc,0},\Pi_\discs,\nabla_\disc,\hbd)$, where the reconstruction function $\Pi_\discs$ is defined by \be\label{def.Pidisc*}
	\forall v_\disc\in X_{\disc,0}\,,\;\forall \cell\in\mesh\,,\;\forall \x\in \cell\,,\\
	\Pi_{\disc^*}v_\disc(\x)=\Pi_\disc v_\disc(\x) + \widetilde{\nabla}_\cell v_\disc\cdot(\x-\x_\cell)
	\ee
	with
	\be
	\widetilde{\nabla}_\cell v_\disc=\frac{1}{|\cell|}\sum_{\edge\in\edgescv}|\edge|v_\edge \: n_{\cell,\edge}.\label{consistentgrad}
	\ee
\end{definition} 
The Hessian scheme corresponding to the modified FVM $B-$HD $\disc^*$ in the sense of Definition \ref{def.modifiedFVM} is given by \eqref{base.HS}, in which only the right-hand side is modified. Thus, the modified FVM has the same matrix as the original FVM. 
\smallskip

Consider now a super-admissible mesh in the sense of \cite[Lemma 13.20]{gdm}, i.e. for $\edge \in \edgesint$ with $\mesh_\edge=\{\cell,L\}$, the straight line $(\x_\cell, \x_L)$ intersects $\edge$ at  $\overline{\x}_\edge$ (similarly on the boundary). This super-admissibility condition is satisfied by rectangles (with $\x_\cell$ the centre of mass of $\cell$) and acute triangles (with $\x_\cell$ the circumcenter of $\cell$).
\begin{proposition}[Superconvergence for modified FVM HD]\label{prop.FVM}
Let $\bu \in H^4(\Omega)\cap H^2_0(\Omega)$ be the solution to \eqref{weak}. Let $u_\discs$ be the solution of the Hessian scheme \eqref{base.HS} for the modified FVM $B-$HD $\discs$ in the sense of Definition \ref{def.modifiedFVM} on a super-admissible mesh. Then for the modified FVM based on $\Delta$-adapted discretisations, there exist a constant $C>0$, independent of $h$, such that
	\begin{equation*}
	\Vert\Pi_\discs u_\discs -\bu\Vert_{}\le C\left\{\begin{array}{ll}h^{1/2}|\ln(h)|^2&\mbox{ if $d=2$}\\
	h^{6/13}&\mbox{ if $d=3$}.\end{array}\right.	
	\end{equation*}
\end{proposition}
Recalling Remark \ref{rates.PDE}, we see that these rates are an improvement over the rates in $H^2$ norm. Precisely, $L^2$ error estimate decays as the square of the $H^2$ error estimate. 
\subsection{Improved $H^1$ error estimate}\label{sec.H1}
To establish an improved $H^1$ error estimate, consider the following dual problem of \eqref{weak}.
\smallskip

The weak formulation for the dual problem with source term $ q\in H^{-1}(\O)$ seeks $\varphi_q \in V$ such that
\begin{equation}\label{weak_adjoint_h1}
a(w,\varphi_q)=(q,w)\; \mbox{for all $w\in V$.}
\end{equation}
Moreover, when $\O$ is convex, $\varphi_q \in H^3(\O)\cap H^2_0(\O)$ with a priori bound $\norm{\varphi_q}{H^3(\O)}\le \norm{q}{H^{-1}(\O)}$ \cite{HBRR}.
In order to state the $H^1$ error estimate, we need to consider the limit-conformity measure between the reconstructed Hessian $\hbd$ and reconstructed gradient $\nabla_\disc$.  Define
\be\label{def.WDtilde}
\begin{aligned}
	&\forall  \chi \in H^B_{\div}(\O)^d\,,\,\widetilde{W}_\disc^B(\chi)=\max_{w_\disc\in X_{\disc,0}\backslash\{0\}}
	\frac{\Big|\widetilde{\mathcal{W}}_\disc^B(\chi, w_\disc) \Big|}{\norm{\hbd w_\disc}{}},
\end{aligned}
\ee
where $H^B_{\div}(\O)^d=\{\chi \in L^2(\O)^{d \times d}:\, \mbox{div}(B^\tau B\chi) \in L^2(\O)^d\}$ and
\be \label{def.WDtilde.nr}
\widetilde{\mathcal{W}}_\disc^B(\chi, w_\disc):=\int_\O \Big( B\chi:\hbd w_\disc+ \mbox{div}(B^\tau B\chi)\cdot\nabla_\disc w_\disc\Big)\d\x .
\ee
Assume the existence of an operator $E_\disc$ which maps the discrete unknowns to the continuous space of functions. This operator plays a central role in the $H^1$ error estimate analysis for HDM. 

\begin{assumption}[Companion operator]\label{def.ed}
	Let $\disc$ be a $B\textendash$Hessian discretisation in the
	sense of Definition \ref{HD}. There exists a linear map $E_\disc :X_{\disc,0} \rightarrow H^2_0(\O)$ called the companion operator. We define 
	\begin{align}
	\omega(E_\disc):=\sup_{\psi_\disc  \in X_{\disc,0}\setminus{\{0\}}}
	\frac{\norm{\nabla_\disc \psi_\disc-\nabla E_\disc \psi_\disc}{}}{\norm{\hbd \psi_\disc}{}}.\label{sup.ED.est}
	\end{align}
\end{assumption}
 Along a sequence of Hessian discretisations $(\disc_m)_{m \in \N}$, it is expected that the companion operators are defined such that $\omega(E_{\disc_m}) \rightarrow 0$ as $m \rightarrow \infty$. 
For example, an explicit companion operator is well-known for the Morley element with $\omega(E_\disc)=\mathcal{O}(h)$ \cite{Morley_plate}.  

\begin{theorem}[Improved $H^1$ error estimate for Hessian schemes] 
	\label{th-h1-super}  \qquad 
	
	\noindent Let $\bu$ be the solution to \eqref{weak}.
	Let $\disc$ be a Hessian discretisation in the sense of Definition \ref{HD} and $u_\disc$ be the solution to the Hessian scheme \eqref{base.HS}. Assume that there exists a companion operator $E_\disc$ in the sense of Assumption \ref{def.ed} and define
	$$q=\frac{-\Delta E_\disc (u_\disc - \P_\disc \bu)}{\norm{\nabla E_\disc (u_\disc- \P_\disc \bu)}{}} \in H^{-1}(\O).$$
	Assume that the solution to \eqref{weak_adjoint_h1} satisfies $\varphi_q \in H^3(\O)\cap H^2_0(\O)$ and choose $\P_\disc \bu,$ $\P_\disc \varphi_q$ $\in X_{\disc,0}$, where $\P_\disc : H^2_0(\O)\rightarrow X_{\disc,0}$. Then 
	\be\nonumber 
	\begin{aligned}
		\Vert\nabla_\disc u_\disc -\nabla \bu\Vert_{}
		\lesssim{}& \big(\omega(E_\disc)+{\widetilde{W}}_\disc^B(\hessian \varphi_q)\big)\big(\WS_\disc^B(\bu)+\norm{\hb\bu - \hbd\P_\disc \bu }{}\big)\\&+ \norm{\nabla\bu - \nabla_\disc\P_\disc \bu}{}
		+|\widetilde{\mathcal{W}}_\disc^B(\hessian \varphi_q, \P_\disc \bu)|\\
		&+\WS_\disc^B(\bu)\norm{\hb \varphi_q -\hbd\P_\disc \varphi_q}{}+|\mathcal{W}_\disc^B(\hessian \bu,\P_\disc \varphi_q )|,
	\end{aligned}
	\ee	
	where $\omega(E_\disc)$ is defined by \eqref{sup.ED.est}, $\WS_\disc^B$ is defined by \eqref{def.ws}, $\mathcal{W}_\disc^B$ is defined by \eqref{def:W.nr}, ${\widetilde{W}}_\disc^B$ is defined by \eqref{def.WDtilde},
	and $\widetilde{\mathcal{W}}_\disc^B$ is defined by \eqref{def.WDtilde.nr}.
\end{theorem}
The proof of Theorem \ref{th-h1-super} is given in Section \ref{sec.proofH1}.
\begin{remark}
	Following Remark \ref{remark.improvedl2.dominate}, for FEMs and methods based on GR operators, Theorem \ref{th-h1-super} gives an improved error estimate in $H^1$ norm if $\norm{\nabla\phi- \nabla_\disc\P_\disc \phi}{}=\mathcal{O}(h^2)$, $\widetilde{\mathcal{W}}_\disc^B(\chi, \P_\disc \phi)=\mathcal{O}(h^2)$, $\omega(E_\disc)=\mathcal{O}(h)$ and ${\widetilde{W}}_\disc^B(\chi)=\mathcal{O}(h)$ for all $\phi \in H^4(\O)\cap H^2_0(\O)$ and all $\chi \in H^1(\O)^{d \times d}$.
\end{remark}
\begin{remark}
	The companion operators actually come with estimates on function, gradient given by \eqref{sup.ED.est} and Hessian (see e.g., \cite{Morley_plate}). The estimates on function and Hessian are not needed in the error analysis and hence we leave them undefined.
\end{remark}
\smallskip

The following proposition talks about the discrete $H^1$ error estimate for lower order conforming and non-conforming FEMs and the proof is given in Section \ref{appendix}, Appendix.
 \begin{proposition}\label{ncfem:supercv.h1}
Let $\bu\in H^4(\O)\cap H^2_0(\O)$ be the solution to \eqref{weak} and $u_\disc$ be the solution to the Hessian scheme \eqref{base.HS}. Then, for low-order conforming FEMs, and Adini and Morley ncFEMs, there exists a constant $C$, not depending on $h$, such that
 	$$\norm{\nabla_\disc u_\disc-\nabla\bu}{}\le Ch^2.$$
 \end{proposition}
 	
	\begin{remark}\label{remark.GRFVMh1}
		The construction of a companion operator $E_\disc$ for the method based on gradient recovery operators  with $\omega(E_\disc)$ small enough is an open problem. Though there is a difficulty of constructing a proper companion operator and hence improved $H^1$ theoretical rate of convergence are not obtained, we observe that the numerical rates in $H^1$ norm are better (see Table \ref{eg.square2}). In numerical test for FVM, the $H^2$ and $H^1$ estimated rates of convergences appear to be both of order 1 (\cite[Section 6]{HDM_linear}). This seems to indicate that we cannot expect an improved estimate in $H^1$ norm compared to the estimate in energy norm. Hence, the FVM method is probably not amenable to an application of Theorem \ref{th-h1-super} (which is an indication that there might not exist, for this method, a proper companion operator).
	\end{remark}

\subsection{Estimates for Morley HDM}\label{sec.Morley}
The following theorem (proof provided in Section \ref{sec.proofMorley}) establishes practical estimates on the quantities \eqref{def.CD}--\eqref{def.WD}. This helps in establishing the convergence of the scheme. 
\begin{theorem}\label{th.Morley}
	Let $\disc$ be a $B$-Hessian discretisation for the Morley element in the sense of Definition \ref{def.Morley}. Then, there exists a constant $C$, not depending on $\disc$, such that
	\begin{itemize}
		\item  $C_\disc^B \le C$,
		\item  $ \forall\: \varphi \in H^3(\O)\cap H^2_0(\O) \quad S_{\disc}^B(\varphi) \le Ch\norm{\varphi}{H^3(\O)}$,
		\item $\forall \:\xi \in H^2({\O})^{d \times d}, \, $ $W_{\disc}^B(\xi)\le Ch\norm{\xi}{H^2({\O})^{d \times d}}.$
	\end{itemize}
\end{theorem}
	The following result is a straightforward consequence of Theorems \ref{th.Morley} and \ref{th:error.est.PDE}.
\begin{corollary}[Convergence]\label{cor:HD.coerciveB}
	Let $(\disc_m)_{m \in \N}$ be a sequence of $B\textendash$Hessian discretisations for the Morley element associated with a mesh $\mesh_m$ such that $h_m\to 0$ as $m\to\infty$, with $B$ satisfying estimate \eqref{coer:B}. Then $\Pi_{\disc_m} u_{\disc_m}\rightarrow \bu$, $\nabla_{\disc_m} u_{\disc_m} \rightarrow \nabla\bu$ and	$\hb_{\disc_m} u_{\disc_m} \rightarrow \hb \bu$ as $m\to\infty$.
\end{corollary}

	\section{ Numerical Results}\label{sec.eg}
	The results of the numerical experiments for the GR method and the modified FVM  are presented in this section. Consider the biharmonic problem $\Delta^2 \bu=f$ on $\O$ with homogeneous clamped boundary conditions. 
	\subsection{Gradient Recovery Method}
Let the relative errors in $L^2(\O)$, $H^1(\O)$ and $H^2(\O)$ norms be denoted by
\begin{align*}
&\err_\disc(\bu):=\frac{\norm{\Pi_\disc u_\disc -\bu}{}}{\norm{\bu}{}},\quad
\err_\disc(\nabla\bu) :=\frac{\norm{\nabla_\disc u_\disc -\nabla\bu}{}}{\norm{\nabla\bu}{}},\\
&
\err_\disc(\hessian\bu) :=\frac{\norm{\nabla Q_h \nabla u_\disc -\hessian\bu}{}}{\norm{\hessian\bu}{}},
\end{align*}
where $u_\disc$ is the solution to the Hessian scheme \eqref{base.HS}. We refer the reader to \cite{BL_stab.mixedfem} for implementation procedure. To determine the effect of the stabilisation function $\stab_h$ on the results, we multiply it by a factor $\rho$ that takes the values 0.001, 1, and 10.
	\subsubsection{Example 1} Let $\O=(0,1)^2$. Figure \ref{femsquaremesh} shows the initial triangulation of a square domain and its uniform refinement. In this example, we choose the right-hand side load function $f$ such that the exact solution is given by $\bu(x,y) =\sin^2(\pi x)\sin^2(\pi y)$. 
	\begin{figure}[ht]
		\centering
		\begin{minipage}[b]{0.45\linewidth}
			\includegraphics[width=12cm]{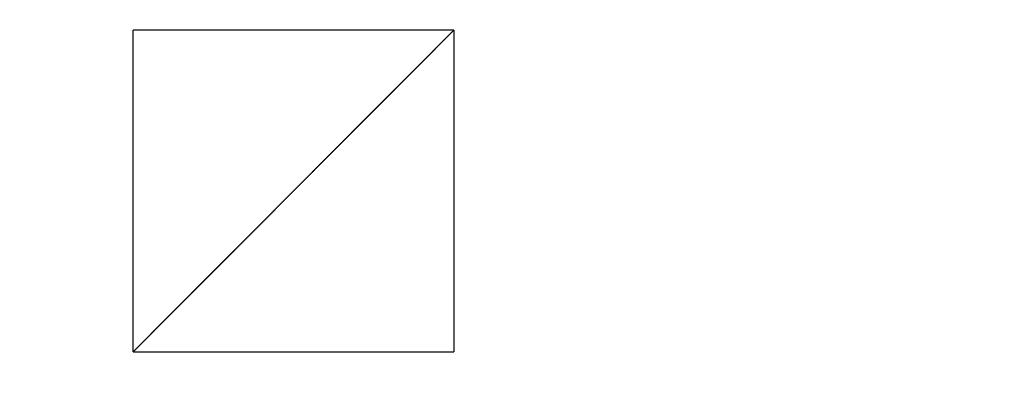}
		\end{minipage}
		\quad
		\begin{minipage}[b]{0.45\linewidth}
			\includegraphics[width=12cm]{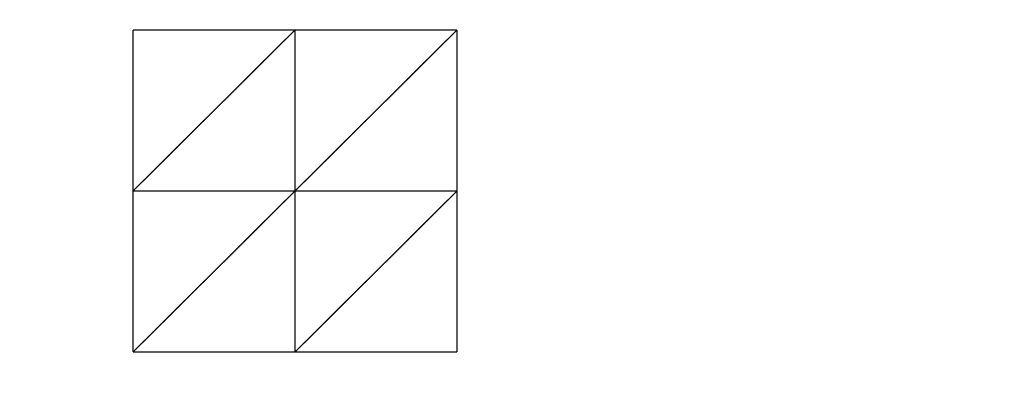}
		\end{minipage}
		\caption{Initial triangulation and uniform refinement of square domain}
		\label{femsquaremesh}
	\end{figure}
The computed errors and orders of convergence in the energy, $H^1$ and $L^2$ norms with $\rho=1$ are shown in Table \ref{eg.square2}. As seen in the table, we obtain linear order of convergence in the energy norm and quadratic order of convergence in $L^2$ norm, which agrees with
the theoretical result in Proposition \ref{ncfem:supercv}. Using gradient recovery operator, a quadratic rate of convergence is obtained in the $H^1$ norm (see Remark \ref{remark.GRFVMh1} for that). 
		
		\begin{table}[h!!]
			\caption{\small{(GR) Convergence results for the relative errors, Example 1, $\rho=1$}}
			{\small{\footnotesize
					\begin{center}
						\begin{tabular}{ ||c||c|c||c|c|| c|c||c|c||c|c||}
							\hline
							$h$ &$\err_\disc(\bu)$ & Order  &  $\err_\disc(\nabla \bu)$ & Order  &$\err_\disc(\hessian \bu)$ & Order  \\ 
							\hline\\[-10.5pt]  &&\\[-9.5pt]
0.353553&3.124409& -    &0.721457& -& 0.855054& -\\
0.176777&0.145381&4.4257&0.099974&2.8513& 0.246640& 1.7936\\
0.088388&0.036224&2.0048&0.023098&2.1138& 0.116470& 1.0824\\
0.044194&0.009068&1.9982&0.005552&2.0566& 0.057308& 1.0232\\
0.022097&0.002261&2.0037&0.001363&2.0266& 0.028470& 1.0093\\
0.011049&0.000564&2.0032&0.000338&2.0116& 0.014198&1.0037\\			
	\hline				
						\end{tabular}
					\end{center}	}}\label{eg.square2}
				\end{table}							
		\subsubsection{Example 2} In this example, we consider the non-convex L--shaped domain given by $\Omega=(-1,1)^2 \setminus\big{(}[0,1)\times(-1,0]\big{)}$. Figure \ref{femLshapemesh} shows the initial triangulation of a L-shaped domain and its uniform refinement. The source term $f$ is chosen such that the model problem has the following exact singular solution \cite{Grisvard92}:
		\begin{align*}
		{u}=(r^2 \cos^2\theta-1)^2 (r^2 \sin^2\theta-1)^2 r^{1+ \gamma}g_{\gamma,\omega}(\theta),
		\end{align*}
		where $(r,\theta)$ denote the polar coordinates, $ \gamma\approx 0.5444837367$ is a non-characteristic 
		root of $\sin^2( \gamma\omega) =  \gamma^2\sin^2(\omega)$, $\omega=\frac{3\pi}{2}$, and $g_{\gamma,\omega}(\theta)=(\frac{1}{\gamma-1}\sin ((\gamma-1)\omega)-\frac{1}{ \gamma+1}\sin(( \gamma+1)\omega))(\cos(( \gamma-1)\theta)-\cos(( \gamma+1)\theta))$ 
		$-(\frac{1}{\gamma-1}\sin(( \gamma-1)\theta)-\frac{1}{ \gamma+1}\sin(( \gamma+1)\theta))
		(\cos(( \gamma-1)\omega)-\cos(( \gamma+1)\omega)).$
		\begin{figure}[ht]
			\centering
			\begin{minipage}[b]{0.45\linewidth}
				\includegraphics[width=12cm]{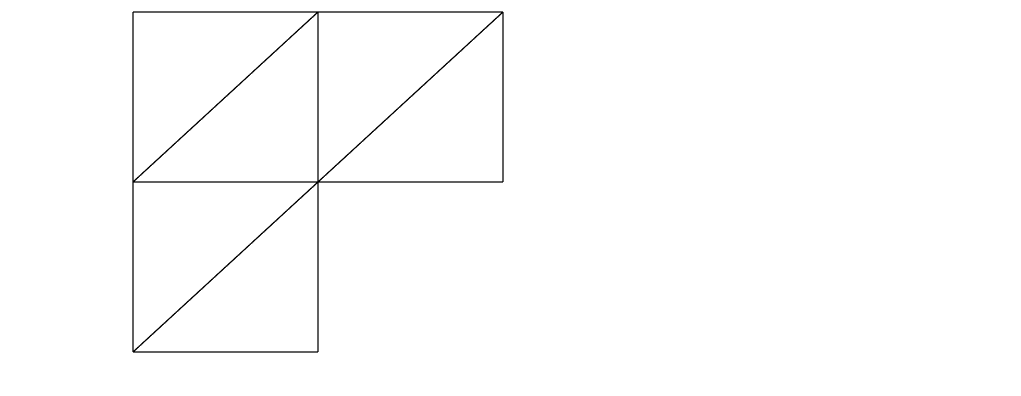}
			\end{minipage}
			\quad
			\begin{minipage}[b]{0.45\linewidth}
				\includegraphics[width=12cm]{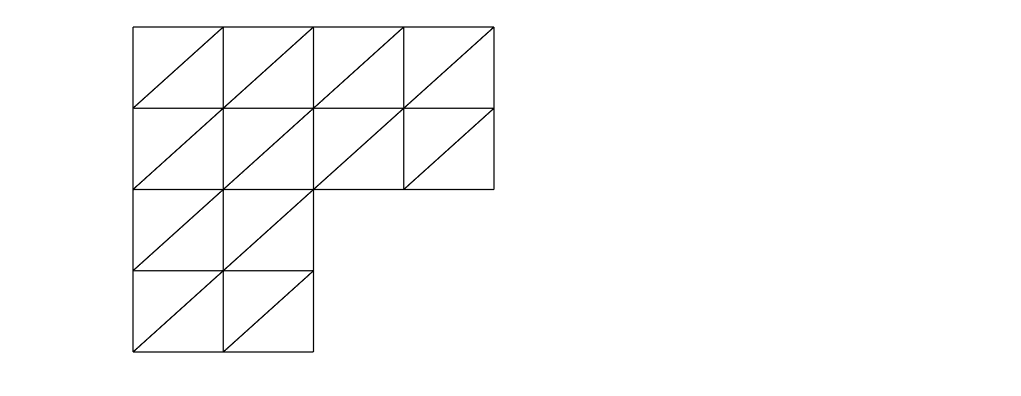}
			\end{minipage}
			\caption{Initial triangulation and uniform refinement of L-shaped domain}
			\label{femLshapemesh}
		\end{figure}
 The errors and rates of convergence are reported in Tables \ref{eg.Lshaped1}--\ref{eg.Lshaped3} respectively. This example is particularly interesting since the solution is less regular due to the corner singularity. The domain $\Omega$ being nonconvex, we expect only suboptimal orders of convergence in the energy, $H^1$ and $L^2$ norms, and this can be clearly seen from the tables. For instance, the convergence rate in $L^2$ norm is 1.5, which is suboptimal. As in Example 1, the numerical rates in $H^1$ norm are similar to those in $L^2$ norm. This improved order of convergence in $H^1$ norm is obtained with the help of gradient recovery operator (see Proposition \ref{ncfem:supercv} and Remark \ref{remark.GRFVMh1}). It can be seen that the stabilisation parameter $\rho$ has a very small impact on the numerical results.

	\begin{table}[h!!]
		\caption{\small{(GR) Convergence results for the relative errors, Example 2, $\rho=0.001$}}
		{\small{\footnotesize
				\begin{center}
					\begin{tabular}{ ||c||c|c||c|c|| c|c||c|c||c|c||}
						\hline
						$h$ &$\err_\disc(\bu)$ & Order  &  $\err_\disc(\nabla \bu)$ & Order  &$\err_\disc(\hessian \bu)$ & Order  \\ 
						\hline\\[-10.5pt]  &&\\[-9.5pt]
0.353553& 1.488937&  -& 0.394870& -&0.504144&-\\
0.176777&0.185753& 3.0028&0.139904& 1.4969& 0.218736& 1.2046\\
0.088388&0.058874&  1.6577&0.045530& 1.6196& 0.116520& 0.9086\\
0.044194&0.018039&1.7065& 0.013756& 1.7267& 0.065220& 0.8372\\
0.022097&0.005400& 1.7401& 0.004197& 1.7128&  0.038827&   0.7483\\
0.011049&0.001681& 1.6835& 0.001396& 1.5882& 0.024390& 0.6707\\
0.005524&0.000570& 1.5617& 0.000526& 1.4085& 0.015899&  0.6174\\
						\hline				
					\end{tabular}
				\end{center}	}}\label{eg.Lshaped1}
			\end{table}	
			
	\begin{table}[h!!]
		\caption{\small{(GR) Convergence results for the relative errors, Example 2, $\rho=1$}}
		{\small{\footnotesize
	\begin{center}
			\begin{tabular}{ ||c||c|c||c|c|| c|c||c|c||c|c||}
						\hline
	$h$ &$\err_\disc(\bu)$ & Order & $\err_\disc(\nabla \bu)$ & Order  &$\err_\disc(\hessian \bu)$ & Order  \\ 
						\hline\\[-10.5pt]  &&\\[-9.5pt]
0.353553&0.447227&    -& 0.377554&  -& 0.441034&  -\\
0.176777&0.177626& 1.3322& 0.142208& 1.4087&  0.217792& 1.0180\\
0.088388&0.059387& 1.5806& 0.046087& 1.6256&  0.115943& 0.9095\\
0.044194&0.018023& 1.7203&0.013886& 1.7307&  0.064817& 0.8390\\
0.022097&0.005360& 1.7496&0.004231& 1.7147& 0.038615& 0.7472\\
0.011049&0.001661& 1.6897&   0.001406&   1.5894& 0.024290& 0.6688\\
0.005524&0.000562&1.5629&   0.000529& 1.4100& 0.015854& 0.6156\\
						\hline				
					\end{tabular}
				\end{center}	}}\label{eg.Lshaped2}
			\end{table}

	\begin{table}[h!!]
\caption{\small{(GR) Convergence results for the relative errors, Example 2, $\rho=10$}}
		{\small{\footnotesize
	\begin{center}
		\begin{tabular}{ ||c||c|c|c|c|| c|c||c|c||c|c||}
										\hline
$h$ &$\err_\disc(\bu)$ & Order  & $\err_\disc(\nabla \bu)$ & Order  &$\err_\disc(\hessian \bu)$ & Order  \\ 
	\hline\\[-10.5pt]  &&\\[-9.5pt]
0.353553&0.488271&  -&0.422393&-& 0.472514&  -   \\
0.176777&0.197355&1.3069&   0.162455&  1.3785&0.226725& 1.0594\\
0.088388&0.064165& 1.6209&0.050639&  1.6817& 0.116820& 0.9567\\
0.044194&0.019077& 1.7500&0.014842& 1.7706& 0.064360& 0.8601\\
0.022097&0.005598& 1.7688&0.0044406&  1.7408& 0.038226&0.7516\\
0.011049&0.001718& 1.7041&0.001455&  1.6102&  0.024090& 0.6662\\
0.005524&0.000576& 1.5759&0.000541& 1.4277& 0.015763& 0.6119\\
										\hline				
									\end{tabular}
								\end{center}	}}\label{eg.Lshaped3}
							\end{table}				
\subsection{Modified Finite Volume Method}
The numerical tests for FVM discussed in Section \ref{FVM} are performed in \cite[Section 6]{HDM_linear}. In this section, three numerical experiments that justify the theoretical result in Proposition \ref{prop.FVM} for modified FVM are presented. We conduct the test on a series of regular triangular meshes (\texttt{mesh1} family) taken from \cite{benchmark} over the unit square $\O=(0,1)^2$. The orthogonality property is satisfied with the point $\x_\cell \in \cell$ chosen as the circumcenter of $K$. Let the relative errors in $L^2(\O)$, $H^1(\O)$ and $H^2(\O)$ norms be denoted by
\begin{align*}
	&\err_\discs(\bu):=\frac{\norm{\Pi_\discs u_\discs -\bu}{}}{\norm{\bu}{}},\quad\err_\discs(\nabla\bu) :=\frac{\norm{\nabla_\disc u_\discs -\nabla\bu}{}}{\norm{\nabla\bu}{}},\\
&	\err_\discs(\Delta\bu) :=\frac{\norm{\Delta_\disc u_\discs -\Delta \bu}{}}{\norm{\Delta\bu}{}},
\end{align*}
where $u_\discs$ is the solution to the Hessian scheme \eqref{base.HS} corresponding to the HD $\discs$ given by Definition \ref{def.modifiedFVM}. 
\subsubsection{Example 1}\label{eg1.FVM}In the first example, choose the right hand side function such that the exact solution is given by $\bu(x,y) = x^2y^2(1-x)^2(1-y)^2$. The error estimates and convergence rates in the energy, $H^1$ and $H^2$ norms are presented in Table \ref{eg1.triangle}. We obtain a quadratic (or slightly better) rate of convergence in $L^2$ norm, linear rate of convergence is $H^1$ norm and sub-linear rate of convergence in $H^2$ norm. Note that the numerical test provides better result compared to the theoretical result, see Proposition \ref{prop.FVM}. The numerical results for modified FVM are similar to those for the FVM. 
	\begin{table}[h!!]
		\caption{\small{(Modified FV) Convergence results, Example 1}}
		{\small{\footnotesize
				\begin{center}
					\begin{tabular}{ ||c||c| c||c| c ||c|c||}
						\hline
						$h$  &$\err_\discs(\bu)$ & Order &$\err_\discs(\nabla\bu)$ & Order  &$\err_\discs(\Delta\bu)$ & Order  \\ 
						\hline\\[-10pt]  &&\\[-10pt]
0.250000& 0.095132& -&   0.236554&   -&   0.134417&-\\
0.125000& 0.024787&1.9403& 0.130595&0.8571&0.068112&0.9807\\
0.062500& 0.005981&2.0511& 0.066013&0.9843&0.038204&0.8342\\
0.031250& 0.001353&2.1442& 0.033053&0.9979&0.022618&0.7562\\
0.015625& 0.000267&2.3415& 0.016526&1.0000&0.014154&0.6763\\
0.007813& 0.000035&2.9347& 0.008262&1.0003&0.009281&0.6089\\
						\hline				
					\end{tabular}
				\end{center}	}} \label{eg1.triangle}
			\end{table}			
\subsubsection{Example 2}\label{eg2.FVM}
In this case, we consider $\bu(x,y) = x^2y^2(1-x)^2(1-y)^2(\cos(2\pi x)+\sin(2\pi y))$. The numerical results, presented in Table \ref{eg2.triangle}, are similar to those obtained for Example 1.		
					
\begin{table}[h!!]
	\caption{\small{(Modified FV) Convergence results, Example 2}}
	{\small{\footnotesize
		\begin{center}
			\begin{tabular}{ ||c||c| c||c| c ||c|c|| }
										\hline
	$h$ &$\err_\discs(\bu)$ & Order &$\err_\discs(\nabla\bu)$ & Order  &$\err_\discs(\Delta\bu)$ & Order  \\ 
		\hline\\[-10pt]  &&\\[-10pt]
   0.250000& 0.230644&     -& 0.458624&  -    & 0.190768&-\\
   0.125000& 0.046952&2.2964& 0.193505& 1.2449& 0.078850&1.2746\\
   0.062500& 0.009022&2.3797& 0.092859& 1.0593& 0.041327&0.9320\\
   0.031250& 0.002089&2.1105& 0.045960& 1.0147& 0.021572&0.9379\\
   0.015625& 0.000502&2.0562& 0.022921& 1.0037& 0.011457&0.9130\\
   0.007813& 0.000120&2.0643& 0.011453& 1.0010& 0.006318&0.8587\\
										\hline				
									\end{tabular}
								\end{center}	}}\label{eg2.triangle}
							\end{table}					
\subsubsection{Example 3}\label{eg3.FVM}
The exact solution is chosen to be $\bu(x,y) =  x^3y^3(1-x)^3(1-y)^3(\exp(x)\sin(2\pi x)+\cos(2\pi x))$. The convergence results are presented in Table \ref{eg3.triangle}. In this example, an $\mathcal{O}(h)$ convergence rate is obtained in $H^2$ norm. Since there is no improvement of the rates from $H^2$ to $H^1,$ as mentioned in Remark \ref{remark.GRFVMh1}, we cannot expect an improved $H^1$ estimate for FVM.
\begin{table}[h!!]
	\caption{\small{(Modified FV) Convergence results, Example 3}}
	{\small{\footnotesize
			\begin{center}
				\begin{tabular}{ ||c||c| c||c| c ||c|c|| }
					\hline
					$h$ &$\err_\discs(\bu)$ & Order &$\err_\discs(\nabla\bu)$ & Order  &$\err_\discs(\Delta\bu)$ & Order  \\ 
					\hline\\[-10pt]  &&\\[-10pt]
0.250000&0.410550&-&0.704301& -&  0.295782&-\\
0.125000&0.029103&3.8183&0.212960&1.7256& 0.084328&1.8104\\
0.062500&0.008773&1.7301&0.096846&1.1368& 0.041288&1.0303\\
0.031250&0.002041&2.1037&0.047833&1.0177& 0.020896&0.9825\\
0.015625&0.000503&2.0203&0.023843&1.0044& 0.010486&0.9947\\
0.007813&0.000125&2.0048&0.011913&1.0011& 0.005249&0.9984\\
					\hline				
				\end{tabular}
			\end{center}	}}\label{eg3.triangle}
		\end{table}	
		\begin{remark}
		For rectangular meshes, in order to satisfy the orthogonality property, $\x_\cell \in \cell$ is chosen as the centre of mass of $\cell$. From \cite[Theorem 5.3]{jd_nn}, it follows that the difference between the source term of modified FVM and original FVM is of $\mathcal{O}(h^2)$. Therefore similar rate of convergence is obtained for modified FVM, since we see an $\mathcal{O}(h^2)$ convergence rate in $L^2$ and $H^1$ norms for FVM in \cite[Section 6]{HDM_linear}.
		\end{remark}								
\section{Proof of the main results}\label{sec.proof} 
The proof of the main results stated in Section \ref{sec.results} are provided in this section. Subsection \ref{sec.proofL2} deals with the proof of improved $L^2$ estimate (Theorem \ref{th-l2-super}) and the proof of improved $H^1$ estimate (Theorem \ref{th-h1-super}) is presented in Subsection \ref{sec.proofH1}. In Subsection \ref{sec.proofMorley}, the estimates associated with the Morley HDM (Theorem \ref{th.Morley}) are derived.
\subsection{Proof of the improved $L^2$ estimate} \label{sec.proofL2}
To prove Theorem \ref{th-l2-super}, we shall make use of the following Lemma, which estimates the error associated with the continuous bilinear form $a(\cdot,\cdot)$ and discrete bilinear form $a_\disc(\cdot,\cdot)$.
\begin{lemma}  \label{conts.discrete.bilinear}
	Let
	$\psi, \phi \in H^2_0(\O)$ be such that $\hessian:A\hessian\psi \in L^2(\O)$
	and $\hessian: A\hessian \phi\in L^2(\O)$.
	Then, for any $\psi_\disc,\phi_\disc\in X_{\disc,0}$, the following holds:
	\begin{align} \label{a_estimate}
	|a(\psi, \phi) - a_\disc(\psi_\disc, \phi_\disc)| \le  \eaaD_\disc(\psi,\phi,\psi_\disc,\phi_\disc),
	\end{align}
	where 
	\begin{align} 
		\eaaD_\disc(\psi,\phi,\psi_\disc,\phi_\disc)  ={}&  |\mathcal{W}_\disc^B(\hessian \psi, \phi_\disc)| +
		|\mathcal{W}_\disc^B(\hessian \phi, \psi_\disc)| +\norm{\Pi_\disc \psi_\disc-\psi}{} \norm{\hessian:A\hessian \phi}{} \nonumber
		\\
		  +\norm{\Pi_\disc \phi_\disc&-\phi}{}  \norm{\hessian:A\hessian \psi}{}  +\norm{\hbd \psi_\disc -\hb \psi}{}\norm{\hbd \phi_\disc - \hb \phi}{}.  \label{def.ED}
	\end{align}
\end{lemma}
\begin{proof} 
	Use the definitions of $a(\cdot,\cdot)$ and $a_\disc (\cdot,\cdot)$ and perform elementary manipulations to obtain
	\begin{align}
	a(\psi, \phi) - a_\disc(\psi_\disc, \phi_\disc)
	& = \int_\O \hb \psi : \hb \phi \d\x - \int_\O \hbd \psi_\disc :\hbd \phi_\disc  \d\x\nonumber\\
	&=\int_\O (\hb \psi-\hbd \psi_\disc) : \hb \phi \d\x\nonumber\\
	&\qquad + \int_\O (\hbd \psi_\disc - \hb \psi) : (\hb \phi -\hbd\phi_\disc) \d\x \nonumber\\
	&\qquad +\int_\O  \hb \psi : (\hb \phi -\hbd \phi_\disc ) \d\x=: T_1 + T_2 + T_3.	\label{a_T}
	\end{align}
	$T_1$ can be estimated using integration by parts twice and \eqref{def:W.nr}.
	\begin{align*} 
	T_1 
	& =  \int_\O \psi (\hessian:A\hessian \phi) \d\x + \mathcal{W}_\disc^B( \hessian \phi, \psi_\disc) 
	-\int_\O  (\hessian:A \hessian \phi)  \Pi_\disc \psi_\disc \d\x.
	\end{align*}
	Hence, by the Cauchy--Schwarz inequality, this gives
	\begin{align} \label{a_T1}
	|T_1| & \le |\mathcal{W}_\disc^B(\hessian \phi, \psi_\disc)|  + \norm{\hessian:A \hessian \phi}{}
	\norm{\psi - \Pi_\disc \psi_\disc}{}.
	\end{align}
	A use of the Cauchy--Schwarz inequality leads to an upper bound for the term $T_2$ as
	\begin{align} \label{a_T2}
	|T_2| & \le  \norm{\hb \psi - \hbd \psi_\disc}{}
	\norm{\hb \phi - \hbd \phi_\disc}{}.
	\end{align}
	The term $T_3$ is estimated exactly as $T_1$ interchanging the roles of $(\psi, \psi_\disc)$ and 
	$(\phi,\phi_\disc)$, which leads to
	\begin{align} \label{a_T3}
	|T_3| & \le |\mathcal{W}_\disc^B(\hessian \psi, \phi_\disc)|  + \norm{\hessian:A \hessian \psi}{}
	\norm{\phi - \Pi_\disc \phi_\disc}{}.
	\end{align}
	A substitution of the estimates \eqref{a_T1}--\eqref{a_T3} into  \eqref{a_T} leads to \eqref{a_estimate}.
\end{proof}
We now prove the main result given by Theorem \ref{th-l2-super}. Note that the proof is obtained by modification of the arguments of \cite[Theorem 3.1]{jd_nn} in the GDM framework to that of HDM.
\begin{proof}[Proof of Theorem \ref{th-l2-super}]
	Choose $w=\bu$ in \eqref{weak_adjoint} and $w_\disc=u_\disc$ in \eqref{adj.GS},
	\be \label{proof.1}
	\Vert \bu-\Pi_\disc u_\disc\Vert_{}=(g, \bu - \Pi_\disc u_\disc)  =a(\bu, \varphi_g) - a_\disc(u_\disc, \varphi_{g,\disc}).
	\ee
 Since $\bu$ and $\varphi_g$ both belong to $H^2_0(\O)$ with $\hessian: A\hessian \bu=f\in L^2(\O)$ and $\hessian: A\hessian \varphi_g=g\in L^2(\O)$,
	a use of \eqref{a_estimate} in \eqref{proof.1} with some manipulations lead to 
	\begin{align}
	\Vert \bu-\Pi_\disc u_\disc\Vert_{}
	&=a(\bu, \varphi_g) -a_\disc(\P_\disc \bu, \P_\disc \varphi_g)+a_\disc(\P_\disc \bu, \P_\disc \varphi_g)- a_\disc(u_\disc, \varphi_{g,\disc})\nonumber\\
	 \le{}& \eaaD_\disc(\bu,\varphi_g,\P_\disc \bu,\P_\disc \varphi_g)+ a_\disc(\P_\disc \bu, \P_\disc \varphi_g) - a_\disc(u_\disc, \varphi_{g,\disc})  \nonumber \\
	 = {}& a_\disc(\P_\disc \bu, \P_\disc \varphi_g -\varphi_{g,\disc}) +  a_\disc(\P_\disc \bu - u_\disc,  \varphi_{g,\disc}) \nonumber \\
	&  +
	\eaaD_\disc(\bu,\varphi_g,\P_\disc \bu,\P_\disc \varphi_g)  =: T_1 + T_2 +  \eaaD_\disc(\bu,\varphi_g,\P_\disc \bu,\P_\disc \varphi_g).
	\label{proof.2}
	\end{align}
 An introduction of $a(\bu,\varphi_g)$, a use of the triangle inequality, \eqref{a_estimate}, \eqref{weak_adjoint} with $ w=\bu$, \eqref{adj.GS} with $w_\disc=\P_\disc \bu$  and the Cauchy--Schwarz inequality yields
	\begin{align} \label{T1}
	|T_{1}| &\le|a(\bu,\varphi_g) -a_\disc(\P_\disc \bu , \varphi_{g,\disc}) |+|a_\disc(\P_\disc \bu , \P_\disc\varphi_g)-a(\bu,\varphi_g) |\nonumber\\
	&\le |a(\bu,\varphi_g) -a_\disc(\P_\disc \bu , \varphi_{g,\disc}) |
	+  \eaaD_\disc(\bu,\varphi_g,\P_\disc \bu,\P_\disc\varphi_g)  \nonumber \\
	&\le  |(g,\bu - \Pi_\disc \P_\disc \bu)| +  \eaaD_\disc(\bu,\varphi_g,\P_\disc \bu,\P_\disc\varphi_g) \nonumber \\
	& \le \norm{g}{}\norm{\bu - \Pi_\disc \P_\disc \bu}{}+ \eaaD_\disc(\bu,\varphi_g,\P_\disc \bu,\P_\disc\varphi_g).
	\end{align}
	We now turn to $T_2$. Introduce the terms $a_\disc(\P_\disc \bu, \P_\disc \varphi_g)$, $a_\disc(u_\disc,\P_\disc \varphi_{g})$ and choose $v_\disc=\P_\disc \varphi_g- \varphi_{g,\disc}$ in \eqref{base.HS} to deduce
	\begin{align} 
	T_2 
	={}& - a_\disc(\P_\disc \bu, \P_\disc \varphi_g-\varphi_{g,\disc}) + a_\disc(u_\disc,  \P_\disc \varphi_g- \varphi_{g,\disc})
	+ a_\disc(\P_\disc \bu-u_\disc, \P_\disc \varphi_g) \nonumber \\
	={}& -\big[a_\disc(\P_\disc \bu, \P_\disc \varphi_g- \varphi_{g,\disc}) -
	(f, \Pi_\disc(\P_\disc \varphi_g- \varphi_{g,\disc}))\big]   + a_\disc(\P_\disc \bu-u_\disc, \P_\disc \varphi_g) \nonumber\\
	={}&-T_{2,1}+T_{2,2}.
	\label{T2}
	\end{align}
	Since $\hessian:A\hessian\bu=f$, \eqref{def:W.nr} yields
	\begin{align*}
	T_{2,1}={}&\int_\Omega \big(\hb\bu: \hbd (\P_\disc \varphi_g- \varphi_{g,\disc}) 
	- f \Pi_\disc(\P_\disc \varphi_g- \varphi_{g,\disc})\big) \d\x\\
	&+ 
	\int_\O( \hbd \P_\disc \bu- \hb\bu):\hbd( \P_\disc \varphi_g - \varphi_{g,\disc})\d\x\\
	={}&\int_\O(\hbd \P_\disc \bu-\hb \bu) :(\hbd \P_\disc \varphi_g - \hbd \varphi_{g,\disc}) \d\x-\mathcal{W}_\disc^B(\hessian\bu,\P_\disc \varphi_g- \varphi_{g,\disc}).
	\end{align*}
	Therefore, apply \eqref{def.WD}, the Cauchy--Schwarz inequality, \eqref{def.ws}, a triangle inequality and \eqref{err.est.PDE} to obtain
	\begin{align}
	|T_{2,1}|\le{}& W_\disc^B(\hessian\bu)\norm{\hbd \P_\disc \varphi_g-\hbd  \varphi_{g,\disc}}{}+\norm{\hbd \P_\disc \bu-\hb \bu}{}
	\norm{\hbd \P_\disc \varphi_g-\hbd  \varphi_{g,\disc}}{}\nonumber\\
	\lesssim{}&\norm{\hbd \P_\disc \varphi_g-\hbd  \varphi_{g,\disc}}{}\big(\WS_\disc^B(\bu)+\norm{\hbd \P_\disc \bu-\hb \bu}{}\big)\nonumber\\
	\lesssim{}&\big(\norm{\hbd \P_\disc \varphi_g-\hb \varphi_g}{}+\WS_\disc^B(\varphi_g)\big)\big(\norm{\hbd \P_\disc \bu-\hb \bu}{}+\WS_\disc^B(\bu)\big).\label{T21}
	\end{align}
	The term $T_{2,2}$ is similar to $T_1$, upon swapping the primal and dual problems, that is $(f,\bu,u_\disc,g,\varphi_g,\varphi_{g,\disc})
	\leftrightarrow (g,\varphi_g,\varphi_{g,\disc},f,\bu,u_\disc)$. Hence, from \eqref{T1},
	\begin{align} \label{T22}
	|T_{2,2}|
	& \le \norm{f}{}\norm{\varphi_g - \Pi_\disc  \P_\disc\varphi_g}{}+ \eaaD_\disc(\bu,\varphi_g,\P_\disc \bu,\P_\disc\varphi_g).
	\end{align}
	Plug the estimates \eqref{T21} and \eqref{T22} in \eqref{T2} to obtain
	\begin{align}
	|T_2|\lesssim{}&\big(\norm{\hbd \P_\disc \bu-\hb \bu}{}+\WS_\disc^B(\bu)\big)\big(\norm{\hbd \P_\disc \varphi_g-\hb \varphi_g}{}+\WS_\disc^B(\varphi_g)\big)\nonumber\\
	&+\norm{f}{}\norm{\varphi_g - \Pi_\disc \P_\disc \varphi_g}{}+ \eaaD_\disc(\bu,\varphi_g,\P_\disc \bu,\P_\disc\varphi_g).
	\label{est.T2}
	\end{align}
	A substitution of \eqref{T1} and 
	\eqref{est.T2} in \eqref{proof.2}
	leads to
	\begin{align*}
	\Vert \bu-\Pi_\disc u_\disc\Vert_{}\lesssim& \big(\norm{\hbd \P_\disc \bu-\hb \bu}{}+\WS_\disc^B(\bu)\big)\big(\norm{\hbd \P_\disc \varphi_g-\hb \varphi_g}{}+
	\WS_\disc^B(\varphi_g)\big)\\
	&+\norm{\bu - \Pi_\disc \P_\disc \bu}{}+\norm{f}{}\norm{\varphi_g - \Pi_\disc \P_\disc \varphi_g}{}+
	\eaaD_\disc(\bu,\varphi_g,\P_\disc \bu,\P_\disc\varphi_g),
	\end{align*}
	where we have used the fact that $\norm{g}{}=1$. Finally, the proof is complete by using the definition \eqref{def.ED} of $\eaaD_\disc$ and noticing that $\hessian: A\hessian \bu=f\in L^2(\O)$ and $\hessian: A\hessian \varphi_g=g\in L^2(\O)$.
\end{proof}

\subsection{Proof of the improved $H^1$ estimate}\label{sec.proofH1}
\begin{proof}[Proof of Theorem \ref{th-h1-super}]
	A use of the triangle inequality leads to
	\be \label{h1.est}
	\norm{\nabla_\disc u_\disc -\nabla \bu}{}\le \norm{\nabla_\disc u_\disc -\nabla_\disc \P_\disc \bu }{}+\norm{\nabla_\disc \P_\disc \bu -\nabla \bu}{}.
	\ee
	Let us estimate $\norm{\nabla_\disc u_\disc -\nabla_\disc \P_\disc \bu}{}$. Set $v_\disc = u_\disc - \P_\disc \bu \in X_{\disc,0}$. Introduce $\nabla E_\disc v_\disc$ and $\hb\bu$, and use triangle inequalities, \eqref{sup.ED.est} and \eqref{err.est.PDE} to obtain
	\begin{align}\label{h1.est1}
	\norm{\nabla_\disc v_\disc }{}&\le  \norm{\nabla_\disc v_\disc-\nabla E_\disc v_\disc}{}+\norm{\nabla E_\disc v_\disc}{}\le \omega(E_\disc) \norm{\hbd v_\disc}{}+\norm{\nabla E_\disc v_\disc}{}\nonumber\\
	&\le  \omega(E_\disc)\big( \norm{\hbd u_\disc - \hb\bu}{}+\norm{\hb\bu - \hbd\P_\disc \bu }{}\big)+\norm{\nabla E_\disc v_\disc}{}\nonumber\\
	&\lesssim  \omega(E_\disc)\big(\WS_\disc^B(\bu)+\norm{\hb\bu - \hbd\P_\disc \bu }{}\big)+\norm{\nabla E_\disc v_\disc}{}.
	\end{align}
	Consider $\norm{\nabla E_\disc v_\disc}{}$. From \eqref{weak_adjoint_h1} with $w=E_\disc v_\disc$, 
	\begin{align}
	\norm{\nabla E_\disc v_\disc}{}=a(E_\disc v_\disc,\varphi_q)&=\int_{\O}(\hb E_\disc v_\disc -\hbd v_\disc): \hb \varphi_q \d\x \nonumber\\
	&\qquad+\int_{\O}\hbd v_\disc: \hb \varphi_q \d\x=:T_1 +T_2.\label{t1t2.est.h1}
	\end{align}
	An integration by parts and a use of \eqref{def.WDtilde.nr}, \eqref{def.WDtilde}, the Cauchy--Schwarz inequality, \eqref{sup.ED.est}, the triangle inequality and \eqref{err.est.PDE} yield
	\begin{align}
	|T_1|
	&\le\int_{\O} |\div(A\hessian \varphi_q)\cdot(\nabla_\disc v_\disc-\nabla E_\disc v_\disc)|\d\x+{\widetilde{W}}_\disc^B(\hessian \varphi_q)\norm{\hbd v_\disc}{}\nonumber\\
	&\le\omega(E_\disc)\norm{\hbd v_\disc}{}\norm{ \div(A\hessian \varphi_q)}{}+{\widetilde{W}}_\disc^B(\hessian \varphi_q)\norm{\hbd v_\disc}{}\nonumber\\
	&\lesssim\big(\omega(E_\disc)\norm{ \div(A\hessian \varphi_q)}{}+{\widetilde{W}}_\disc^B(\hessian \varphi_q)\big)\big(\WS_\disc^B(\bu)+\norm{\hb\bu - \hbd\P_\disc \bu }{}\big).\label{t1.est.h1}
	\end{align}
 Simple manipulations leads to
	\begin{align}
	T_2
	&=\int_\O (\hb \bu-\hbd \P_\disc \bu) : \hb \varphi_q \d\x+ \int_\O (\hbd u_\disc-\hb \bu ) : (\hb \varphi_q -\hbd\P_\disc \varphi_q) \d\x \nonumber\\
	&\qquad +\int_\O  (\hbd u_\disc -\hb \bu ) :\hbd \P_\disc \varphi_q\d\x=: T_{2,1} + T_{2,2}+ T_{2,3}.\label{t21-t23.est.h1}
	\end{align}
Integration by parts, \eqref{def.WDtilde.nr} and the Cauchy--Schwarz inequality imply that
	\be
	|T_{2,1}|\le \norm{\div (A\hessian \varphi_q)}{}\norm{\nabla_\disc\P_\disc \bu-\nabla\bu}{}+|\widetilde{\mathcal{W}}_\disc^B(\hessian \varphi_q, \P_\disc \bu)|.\label{t21.est.h1}
	\ee
	Apply Cauchy--Schwarz inequality and \eqref{err.est.PDE} to obtain
	\be
	|T_{2,2}|\le \norm{\hb \bu - \hbd u_\disc}{}\norm{\hb \varphi_q -\hbd\P_\disc \varphi_q}{}\lesssim \WS_\disc^B(\bu)\norm{\hb \varphi_q -\hbd\P_\disc \varphi_q}{}.\label{t22.est.h1}
	\ee
	Since $\hessian:A\hessian\bu=f$, by \eqref{def:W.nr} and \eqref{base.HS} with $v_\disc=\P_\disc \varphi_q$, the term $T_{2,3}$ can be estimated as
	\begin{align}
	T_{2,3}
	&\le-\int_{\O}(\hessian:A\hessian \bu)\Pi_\disc \P_\disc \varphi_q\d\x +\mathcal{W}_\disc^B(\hessian \bu,\P_\disc \varphi_q )+\int_\O  \hbd u_\disc :\hbd \P_\disc \varphi_q\d\x\nonumber\\
	&=-\int_{\O}(\hessian:A\hessian \bu)\Pi_\disc \P_\disc \varphi_q\d\x +\mathcal{W}_\disc^B(\hessian \bu,\P_\disc \varphi_q )+\int_{\O}f\Pi_\disc \P_\disc \varphi_q\d\x\nonumber\\
	&=\mathcal{W}_\disc^B(\hessian \bu,\P_\disc \varphi_q). \label{t23.est.h1}
	\end{align}
	A substitution of \eqref{t21.est.h1}--\eqref{t23.est.h1} in \eqref{t21-t23.est.h1} yields
	\begin{align}
	|T_2|&\lesssim \norm{\div (A\hessian \varphi_q)}{}\norm{\nabla\bu - \nabla_\disc\P_\disc \bu}{}+|\widetilde{\mathcal{W}}_\disc^B(\hessian \varphi_q, \P_\disc \bu)|\nonumber\\
	&\qquad +\WS_\disc^B(\bu)\norm{\hb \varphi_q -\hbd\P_\disc \varphi_q}{}+|\mathcal{W}_\disc^B(\hessian \bu,\P_\disc \varphi_q )|.
	\label{t2.est.h1}
	\end{align}
	Plug \eqref{t1.est.h1} and \eqref{t2.est.h1} in  \eqref{t1t2.est.h1} to obtain an estimate for $\norm{\nabla E_\disc v_\disc}{}$ as
	\begin{align}
	\norm{\nabla E_\disc v_\disc}{} &\lesssim \big(\omega(E_\disc)\norm{ \div(A\hessian \varphi_q)}{}+{\widetilde{W}}_\disc^B(\hessian \varphi_q)\big)\big(\WS_\disc^B(\bu)+\norm{\hb\bu - \hbd\P_\disc \bu }{}\big)\nonumber\\
	&\quad+ \norm{\div (A\hessian \varphi_q)}{}\norm{\nabla\bu - \nabla_\disc\P_\disc \bu}{}+|\widetilde{\mathcal{W}}_\disc^B(\hessian \varphi_q, \P_\disc \bu)|\nonumber\\
	&\qquad +\WS_\disc^B(\bu)\norm{\hb \varphi_q -\hbd\P_\disc \varphi_q}{}+|\mathcal{W}_\disc^B(\hessian \bu,\P_\disc \varphi_q )|.
	\end{align}
	A use of the apriori bound for the dual problem $\norm{\varphi_q}{H^3(\O)}\lesssim 1 $ yields
	\begin{align}
	\norm{\nabla E_\disc v_\disc}{} \lesssim &\big(\omega(E_\disc)+{\widetilde{W}}_\disc^B(\hessian \varphi_q)\big)\big(\WS_\disc^B(\bu)+\norm{\hb\bu - \hbd\P_\disc \bu }{}\big)\nonumber\\
	&+ \norm{\nabla\bu - \nabla_\disc\P_\disc \bu}{}+|\widetilde{\mathcal{W}}_\disc^B(\hessian \varphi_q, \P_\disc \bu)|\nonumber\\
	&+\WS_\disc^B(\bu)\norm{\hb \varphi_q -\hbd\P_\disc \varphi_q}{}+|\mathcal{W}_\disc^B(\hessian \bu,\P_\disc \varphi_q )|.\label{ed.est}
	\end{align}
	A substitution of \eqref{ed.est} into \eqref{h1.est1} leads to an estimate on $\norm{\nabla_\disc v_\disc}{}$ (with $v_\disc = u_\disc - \P_\disc \bu \in X_{\disc,0}$) which when plugged on \eqref{h1.est} gives
	\begin{align*}
	\norm{\nabla_\disc u_\disc -\nabla \bu }{}&\lesssim \big(\omega(E_\disc)+{\widetilde{W}}_\disc^B(\hessian \varphi_q)\big)\big(\WS_\disc^B(\bu)+\norm{\hb\bu - \hbd\P_\disc \bu }{}\big)\nonumber\\
	&\quad+ \norm{\nabla\bu - \nabla_\disc\P_\disc \bu}{}+|\widetilde{\mathcal{W}}_\disc^B(\hessian \varphi_q, \P_\disc \bu)|\nonumber\\
	&\qquad +\WS_\disc^B(\bu)\norm{\hb \varphi_q -\hbd\P_\disc \varphi_q}{}+|\mathcal{W}_\disc^B(\hessian \bu,\P_\disc \varphi_q )|
	\end{align*}
	and this completes the proof.
\end{proof}	
	\subsection{Proof of the HDM properties for the Morley element}\label{sec.proofMorley}
	\begin{proof}[Proof of Theorem \ref{th.Morley}]
		Let $\disc=(X_{\disc,0},\Pi_\disc,\nabla_\disc,\hbd)$ be a $B\textendash$Hessian discretisation for the Morley ncFEM in the sense of Definition \ref{def.Morley}. In the sequel, we will use a generic constant $C$, which will take different values at different places but will always be independent of the mesh size $h$.

		$\bullet$ \textbf{Coercivity: }	Let $v_\disc \in X_{\disc,0}$. Since $\llbracket \Pi_\disc v_\disc \rrbracket$ = 0 at the face vertices for any $v_\disc \in X_{\disc,0}$ and $\llbracket \nabla_\disc v_\disc \rrbracket=0$ at the edge midpoints, use Lemma \ref{jump_bb_gradient} twice and the coercivity property of $B$ given by \eqref{coer:B} to obtain
		\be \nonumber
		\norm{\Pi_\disc v_\disc}{} \le C \norm{\nabla_{\disc}v_\disc}{} \le C \norm{\hessian_{\disc}v_\disc}{}\le C\varrho^{-1}\norm{\hbd v_\disc}{}.
		\ee
		This with \eqref{def.CD} concludes the estimate on $C_\disc^B$.
		
		$\bullet$ \textbf{Consistency:} Consistency follows from the interpolation property of the family of Morley element \cite[Chapter 6]{ciarlet1978finite}. For $\varphi \in H^3(\O)\cap H^2_0(\O)$,
		\begin{align*}
		\inf_{w_\disc\in X_{\disc,0}} \norm{\Pi_\disc w_\disc- \varphi}{} & \le Ch^3\lVert \varphi \rVert_{H^3(\O)}, \quad \inf_{w_\disc\in X_{\disc,0}} \norm{\nabla_\disc w_\disc -\nabla \varphi }{} \le Ch^2\lVert \varphi \rVert_{H^3(\O)}, \\
		\inf_{w_\disc\in X_{\disc,0}} \norm{\hbd& w_\disc-\hessian^B \varphi}{} \le Ch\lVert \varphi \rVert_{H^3(\O)}.
		\end{align*}
		Therefore, we obtain $S_\disc^B(\varphi)\le C h \lVert \phi \rVert_{H^3(\O)}$.
		\smallskip
		
		$\bullet$ \textbf{Limit--conformity:} For any $\xi \in H^2({\O})^{d \times d}$ and $v_\disc \in X_{\disc,0}$, cellwise integration by parts yields
		\begin{align*}
		\int_{\O}^{}(\hessian:A\xi)\Pi_\disc v_\disc \d\x&=\sum_{\cell \in \mesh}^{} \int_{\cell}^{}(\hessian:A\xi)\Pi_\disc v_\disc \d\x\\
		&=\int_{\O}^{}A\xi:\hessian_\disc v_\disc \d\x-\sum_{\cell \in \mesh}^{}
		\int_{\partial \cell}^{}(A\xi n_\cell)\cdot\nabla_\disc v_\disc \d s(\x) \\
		& \qquad + \sum_{\cell \in \mesh}^{} \int_{\partial \cell}^{}(\div(A\xi)\cdot n_\cell)\Pi_\disc v_\disc \d s(\x).
		\end{align*}
This gives
		\begin{align}
		\int_{\O}^{}(\hessian:A\xi)\Pi_\disc v_\disc \d\x	-\int_{\O}^{}A\xi:&\hessian_\disc v_\disc \d\x
		=-\sum_{\cell \in \mesh}^{}\sum_{\edge \in \edgescv}^{} 
		\int_{\edge}^{}(A\xi n_{\cell,\edge})\cdot \nabla_\disc v_\disc  \d s(\x)\nonumber\\+ &\sum_{\cell \in \mesh}^{}\sum_{\edge \in \edgescv}^{} \int_{\edge}^{}(\div(A\xi) \cdot n_{\cell,\edge})\Pi_\disc v_\disc \d s(\x).\label{limit_conformity1}
		\end{align}
An appropriate modification to the proof of \cite[Lemma 3.5]{PLPL75} yields
		\be \nonumber 
		\bigg| \int_{\O}^{}(\hessian:A\xi)\Pi_\disc v_\disc \d\x-\int_{\O}^{}A\xi:\hessian_\disc v_\disc \d\x\bigg|\le C\varrho^{-1}h\norm{\hessian_\disc^B {v_\disc}}{}\norm{\xi}{H^2({\O})^{d \times d}}
		\ee 
		and this leads to the desired estimate on $W_\disc^B$. 
	\end{proof}

	\section{Appendix}\label{appendix}
	In this section, the proofs of Propositions \ref{ncfem:supercv}, \ref{prop.FVM} and \ref{ncfem:supercv.h1} are presented. This is followed by some technical results.
	\subsection{Proof of the applications of improved $L^2$ error estimate}
We start by a preliminary result that states the approximation properties of the classical interpolant $\P_\disc$ for various methods.
\begin{lemma}[Interpolation \cite{ciarlet1978finite}, \cite{HDM_linear}]\label{lemma.interpolant}
	Let $\psi\in H^3(\O)\cap H^2_0(\O)$ and $\phi\in H^4(\O)\cap H^2_0(\O)$. The classical interpolant satisfies 
	\item[(i)] For conforming FEMs and Morley ncFEM, 
	 \be \nonumber
	 \norm{\Pi_\disc\P_\disc\psi-\psi}{}\le Ch^3, \norm{\nabla_\disc\P_\disc\psi-\nabla\psi}{}\le Ch^2 \mbox{ and }\norm{\hbd\P_\disc\psi-\hb\psi}{}\le Ch.
	 \ee 
	 \item[(ii)]For Adini ncFEM,
	 \be \nonumber 
	 \norm{\Pi_\disc\P_\disc\phi-\phi}{}\le Ch^4, \norm{\nabla_\disc\P_\disc\phi-\nabla\phi}{}\le Ch^3 \mbox{ and }\norm{\hbd\P_\disc\phi-\hb\phi}{}\le Ch^2.
	 \ee
	 \item[(iii)] For the methods based on gradient recovery operator,
	 	$$\norm{\Pi_\disc\P_\disc\psi-\psi}{}\le Ch^2, \norm{\nabla_\disc\P_\disc\psi-\nabla\psi}{}\le Ch^2 \mbox{ and }\norm{\hbd\P_\disc\psi-\hb\psi}{}\le Ch. $$
\end{lemma}
The next lemma establishes an estimate on the limit--conformity measure $\mathcal{W}_\disc^B$ given by \eqref{def:W.nr} for various schemes.
\begin{lemma}\label{lemma.limitconformity}
		Let $\xi \in H^2(\O)^{d \times d}$, $\psi\in H^3(\O)\cap H^2_0(\O)$ and $\phi\in H^4(\O)\cap H^2_0(\O)$.
			\item[(i)] For conforming FEMs, we have $\mathcal{W}_\disc^B(\xi,\P_\disc\psi)=0$.
			\item[(ii)] For Adini ncFEM, $\mathcal{W}_\disc^B(\xi,\P_\disc\phi)=\mathcal{O}(h^2)$.
		\item[(iii)] For Morley ncFEM and gradient recovery methods, $\mathcal{W}_\disc^B(\xi,\P_\disc\psi)=\mathcal{O}(h^2)$.	
\end{lemma}		
\begin{proof}
$(i)\textsc{ Conforming FEMs. }$Since $X_{\disc,0} \subseteq H^2_0(\O)$, using integration by parts twice, the limit-conformity measure vanishes, that is, $\mathcal{W}_\disc^B=0$.

		\smallskip
		
		 	$(ii)$ \textsc{Nonconforming FEM: the Adini rectangle. }Let $\phi\in H^4(\O)\cap H^2_0(\O)$ and $\xi \in H^2(\O)^{d \times d}$. Introduce the term $(\hessian:A\xi)\phi$ in \eqref{def:W.nr}, use the Cauchy--Schwarz inequality and Lemma \ref{lemma.interpolant} to obtain
		 	\begin{align*}
		 	\big|\mathcal{W}_\disc^B(\xi,\P_\disc\phi)\big|
		 	&\le\bigg| \int_\O \big((\hessian:A\xi)\Pi_\disc \P_\disc \phi -(\hessian:A\xi)\phi\big)\d\x\bigg|\\
		 	&\qquad \qquad+ \bigg|\int_{\O}\big((\hessian:A\xi)\phi-B\xi:\hbd \P_\disc \phi \big)\d\x\bigg|\\
		 		&\le \norm{\hessian:A\xi}{}\norm{\Pi_\disc \P_\disc \phi-\phi}{}+\bigg|\int_{\O}\big((\hessian:A\xi)\phi-B\xi:\hbd \P_\disc \phi \big)\d\x\bigg|\\
		 	&\le Ch^4+ \bigg|\int_{\O}\big((\hessian:A\xi)\phi-B\xi:\hbd \P_\disc \phi \big)\d\x\bigg|.
		 	\end{align*}
		 	Apply integration by parts twice to deduce
		 	\begin{align} \label{est.WDtilda1}
		 	\big|\mathcal{W}_\disc^B(\xi,P_\disc\phi)\big|&\le Ch^4+\bigg|\int_{\O}\big(B\xi:\hb \phi-B\xi:\hbd \P_\disc \phi \big)\d\x\bigg|.
		 	\end{align}
	A use of the Cauchy--Schwarz inequality and Lemma \ref{lemma.interpolant} leads to
	 	 \begin{align} 
	 	 \big|\mathcal{W}_\disc^B(\xi,P_\disc\phi)\big|
	 	 &\le Ch^4 + \norm{B\xi}{}\norm{\hbd\P_\disc\phi-\hb\phi}{}\le Ch^2.\nonumber
	 	 \end{align}	
		 	\smallskip
		 	
	$(iii)(a)$ \textsc{Nonconforming FEM: the Morley triangle. }Let $\psi\in H^3(\O)\cap H^2_0(\O)$ and $\xi \in H^2(\O)^{d \times d}$. Proceeding as in the proof of limit conformity $\mathcal{W}_\disc^B(\xi,\P_\disc\psi)$ for the Adini's rectangle (with $\norm{\Pi_\disc\P_\disc\psi-\psi}{}\le Ch^3$), from \eqref{est.WDtilda1}, we arrive at
		 	\begin{align}\label{Morley.Pd1}
		 \mathcal{W}_\disc^B(\xi,\P_\disc\psi)&\le Ch^3 + \bigg|\int_{\O}\big(B\xi:\hb \psi-B\xi:\hbd \P_\disc \psi \big)\d\x\bigg|.
		 	\end{align}
	Let $\xi_\cell$ be the average value of $\xi$ on the cell $\cell \in \mesh$.  By the mesh regularity assumption, $\norm{\xi-\xi_\cell}{L^2(\cell)^{d \times d}}\le Ch\norm{\xi}{H^1(\cell)^{d \times d}}$ (see, e.g., \cite[Lemma B.6]{gdm}). An introduction of $B\xi_\cell$ in the above inequality and a use of the Cauchy--Schwarz inequality and Lemma \ref{lemma.interpolant} yield
			 	\begin{align}
			 	\big|\mathcal{W}_\disc^B(\xi,\P_\disc\psi)\big|
			& \le Ch^3+	\sum_{\cell \in \mesh}^{}\norm{B\xi-B\xi_\cell}{L^2(\cell)^{d \times d}}\norm{\hbd \psi-\hbd \P_\disc \psi}{L^2(\cell)^{d \times d}}\nonumber\\
			&\qquad \qquad +\bigg|\sum_{\cell \in \mesh}^{}\int_{\cell}B\xi_\cell:(\hb\psi-\hbd \P_\disc \psi)\d\x\bigg|\nonumber\\
			 	&\le Ch^2+\bigg|\sum_{\cell \in \mesh}^{}\int_{\cell}B\xi_\cell:(\hb\psi-\hbd \P_\disc \psi) \d\x\bigg|.\nonumber 
			 	\end{align}	 
	For $\cell \in \mesh$, we have \cite{Gallistl_Morley}
	 \be \label{Morley.Pd}
	  \int_{\cell}\hbd \P_\disc \psi\d\x=\int_{\cell}\hb \psi\d\x.
	 \ee	 		
	 Hence, $\mathcal{W}_\disc^B(\xi,\P_\disc\psi)=\mathcal{O}(h^2)$. 
		 	
		 	\smallskip
		 	
	$(iii)(b)$ \textsc{Gradient recovery method. }
	Note that for the GR method, $\Pi_\disc\P_\disc\psi=\P_\disc\psi \in V_h$, an $H^1_0$-conforming finite element space which contains the piecewise linear functions, and $\norm{\nabla \P_\disc\psi-\nabla\psi}{}\le Ch$. Let us consider $\mathcal{W}_\disc^B(\xi,\P_\disc\psi)$. Reproducing the same steps as in the proof for Adini's rectangle (with $\norm{\Pi_\disc\P_\disc\psi-\psi}{}\le Ch^2$), from \eqref{est.WDtilda1} and the definition of reconstructed Hessian $\hbd$ (see Section \ref{sec.HDMeg}), we obtain 
		 	\begin{align*}
		 	\big|\mathcal{W}_\disc^B(\xi,\P_\disc\psi)\big|
		 	&\le Ch^2+\bigg|\int_{\O}\Big(A\xi:\hessian\psi-A\xi:\nabla Q_h \nabla \P_\disc \psi\Big)\d\x\bigg|\\
		 	&\qquad +\bigg|\int_{\O}A\xi:(\stab_h\otimes (Q_h\nabla \P_\disc \psi-\nabla \P_\disc \psi) )\d\x\bigg|=: Ch^2+A_1 + A_2.
		 	\end{align*}
		 	Since $Q_h\nabla \P_\disc \psi \in H^1_0(\O)$, an integration by parts, the Cauchy--Schwarz inequality and the approximation property of $\P_\disc$ given by Lemma \ref{lemma.interpolant} show that
		 	\begin{align}
		 	|A_1|&=\Big\lvert-\int_\O\nabla \psi \cdot\div (A\xi)\d\x +\int_\O Q_h \nabla \P_\disc \psi\cdot\mbox{div}(A\xi)\d\x\Big\rvert \nonumber\\
		 	&\le\norm{Q_h \nabla \P_\disc \psi-\nabla \psi}{}\norm{\mbox{div}(A\xi)}{}=\norm{\nabla_\disc \P_\disc \psi-\nabla \psi}{}\norm{\mbox{div}(A\xi)}{}\le Ch^2.\nonumber
		 	\end{align}
		 	To estimate $A_2$, we shall make use of the orthogonality property of the stabilisation function. For all $\cell \in \mesh$, 
		 	denoting by $V_h(\cell)=\{v_{|\cell}\,:\,v \in V_h\,,\; \cell \in \mesh\}$ the local finite element space,
		 	$$\left[\stab_{h|K}\otimes (Q_h\nabla-\nabla)(V_h(\cell))\right] \perp \nabla V_h(\cell)^d$$ where the orthogonality is understood in $L^2(K)^{d\times d}$ with the inner product
		 	induced by ``$:$''. Let $\xi_\cell$ denote the average of $\xi$ over $\cell \in \mesh$. Since the finite dimensional space $V_h$ contains the piecewise linear functions, $\nabla V_h(K)$ contains
		 	the constant vector-valued functions on $K$ and thus, by the orthogonality condition, the Cauchy--Schwarz inequality, the boundedness of $\stab_h$, the triangle inequality and the approximation properties of the interpolant,
		 	\begin{align*}
		 	|A_2|
		 	={}&\Big\lvert\sum_{\cell \in \mesh}^{}\int_\cell(A\xi-A\xi_\cell):\stab_h\otimes (Q_h\nabla \P_\disc \psi-\nabla \P_\disc \psi)\d\x\Big\rvert\nonumber\\
		 	\le{}& C \sum_{\cell \in \mesh}^{}\norm{\xi-\xi_\cell}{L^2(\cell)^{d \times d}}\norm{Q_h\nabla \P_\disc \psi-\nabla \P_\disc \psi}{L^2(\cell)^d}\nonumber\\
		 	\le{}& Ch\norm{\nabla_\disc \P_\disc \psi-\nabla \P_\disc \psi}{}\nonumber\\
		 	\le{}& Ch\Big(\norm{\nabla_\disc \P_\disc \psi-\nabla \psi}{}+\norm{\nabla \psi-\nabla \P_\disc \psi}{}\Big)\le Ch^2.
		 	\end{align*}
		 	Therefore, we obtain
		 	$\mathcal{W}_\disc^B(\xi,\P_\disc\psi)=\mathcal O(h^2).$
		 \end{proof}
	\begin{proof}[Proof of Proposition \ref{ncfem:supercv}]
		The proof of Proposition \ref{ncfem:supercv} follows from Theorem \ref{th:error.est.PDE}, Remark \ref{rates.PDE}, Lemma \ref{lemma.interpolant} and Lemma \ref{lemma.limitconformity}.
		\end{proof}	 	
\begin{proof}[Proof of Proposition \ref{prop.FVM}]
    As a consequence of Stokes' formula, we have for $\cell \in \mesh,$ $\sum_{\edge\in\edgescv}|\edge|\:n_{\cell,\edge}=0$ (see the proof of \cite[Lemma B.3]{gdm}). A use of \eqref{vsigma} and the superadmissible mesh condition $n_{\cell,\edge}=\frac{\overline{\x}_\edge-\x_\cell}{d_{\cell,\edge}}$ leads to
	    \be \nonumber
	    \widetilde{\nabla}_\cell v_\disc=\frac{1}{|\cell|}\sum_{\edge\in\edgescv}|\edge|(v_\edge-v_\cell) \: n_{\cell,\edge}=\nabla_\cell v_\disc,
	    \ee
	    where $(\nabla_\disc v_\disc)_\cell=\nabla_\cell v_\disc$ as defined in Section \ref{FVM}. Hence,
	   $$\int_\cell \nabla_\disc v_\disc\d\x=\int_\cell \nabla_K v_\disc \d\x=|\cell|\widetilde{\nabla}_\cell v_\disc.$$
	  The definition of $D^*$, the above relation between $\widetilde{\nabla}_\cell$ and $\nabla_\disc,$ and \eqref{def.CD} imply
	   \be\nonumber 
	   \forall v_\disc\in X_{\disc,0}\,,\;\norm{\Pi_\disc v_\disc-\Pi_{\discs}v_\disc}{L^2(\O)}\lesssim
	   h\norm{\hbd v_\disc}{}.
	   \ee
	   Therefore, following the proof of \cite[Remark 7.51]{gdm}, we obtain the same estimates on $C_\discs^B$, $S_\discs^B$ and $W_\discs^B$ for $\discs$ as that for the original FVM HD $\disc$ given in Section \ref{FVM}. Thus, from Remark \ref{rates.PDE}, under regularity assumption, an $\mathcal{O}(h^{1/4}|\ln(h)|)$ (in $d=2$) or $\mathcal{O}(h^{3/13})$ (in $d=3$) error estimate can be obtained for the Hessian scheme based on modified FVM HD $\discs$. Note that to prove the error estimates for original FVM, the interpolation $\P_\disc$ is constructed by solving a TPFA scheme for second order problem, i.e, by considering
	  $ |\cell|\Delta_\cell \P_\disc \phi=\int_\cell\Delta \phi\d\x $ for $ \phi$  smooth enough and $\cell \in \mesh$. To preserve a superconvergence for this modified FVM, the idea is to construct $\P_\discs \phi$ by solving the modified TPFA scheme, where $\Pi_\disc$ is replaced by $\Pi_\discs.$ Since TPFA and Hybrid Mimetic Mixed (HMM) schemes are the same on superadmissible meshes, from \cite[Theorem 4.6]{jd_nn},
	  \be\label{interpolationFVM}
	  \norm{\Pi_{\discs} \P_\discs \phi-\phi}{}\lesssim h^2\norm{\phi}{H^2(\O)}.
	  \ee
	To estimate $\mathcal{W}_\discs^B(\xi,\P_\discs\phi)$, for $\phi \in H^4(\O)\cap H^2_0(\Omega)$ and $\xi \in H^2(\O)^{d \times d},$ consider \eqref{def:W.nr} with $\disc=\discs$. Introduce $(\hessian:A\xi)\phi$, use the Cauchy-Schwarz inequality, \eqref{interpolationFVM} and integration by parts twice to obtain
	\begin{align*}
\big|\mathcal{W}_\discs^B(\xi,\P_\discs\phi)\big|
&\le \bigg|\int_\O \Big((\hessian:A\xi)(\Pi_\discs \P_\discs\phi -\phi) \Big)\d\x\bigg|\\
&\qquad + \bigg|\int_\O \Big((\hessian:A\xi)\phi - B\xi:\hbd \P_\discs\phi \Big)\d\x\bigg|\\
&\le Ch^2 +\bigg|\int_\O B\xi:(\hb \phi-\hbd \P_\discs\phi) \d\x\bigg|.
	\end{align*}
The second term on the right-hand side of the above inequality can be estimated by considering the projection of  $B\xi$ on piecewise constant functions on the mesh $\mesh$. Let $B\xi_\cell$ be the projection of $B\xi$ on $\cell \in \mesh.$ Since $\Delta_\disc \P_\discs \phi$ is the projection of $\Delta \phi$ on piecewise constant functions on $\mesh$ (that is, $|\cell|\Delta_\cell \P_\discs \phi=\int_\cell\Delta \phi\d\x$), a use of the orthogonality property of the projection operator, the Cauchy-Schwarz inequality and the approximation property leads to
\begin{align}\nonumber
\big|\mathcal{W}_\discs^B(\xi,\P_\discs\phi)\big|&\le Ch^2 +\bigg|\sum_{\cell \in \mesh}^{}\int_\mesh (B\xi-B\xi_\cell):(\hb \phi-\hbd \P_\discs\phi) \d\x\bigg|\le Ch^2.
\end{align}
 A substitution of the above estimate, \eqref{interpolationFVM} and estimates given by Remark \ref{rates.PDE} in Theorem \ref{th-l2-super} with $\disc=\discs$ yields the desired estimate.
		 \end{proof}
		\subsection{Proof of the applications of improved $H^1$ error estimate}
		\begin{proof}[Proof of Proposition \ref{ncfem:supercv.h1}]
	$\bullet$ \textsc{Conforming FEMs. }Let $\psi\in H^3(\O)\cap H^2_0(\O)$. Since $X_{\disc,0} \subseteq H^2_0(\O)$, by applying integration by parts, the measure of limit-conformity ${\widetilde{W}}_\disc^B$ vanishes. Also, companion operator $E_\disc$ is nothing but the identity operator which implies $\omega(E_\disc)=0$. Hence, under regularity assumption on $\bu$, combine these estimates along with Remark \ref{rates.PDE}, Lemma \ref{lemma.interpolant} and Lemma \ref{lemma.limitconformity} in Theorem \ref{th-h1-super} to obtain $\norm{\nabla_\disc u_\disc-\nabla\bu}{}\le Ch^2$.
 \smallskip
 
 	$\bullet$ \textsc{Non-conforming FEM: the Adini rectangle.}
		 	The estimate $\omega(E_\disc)=\mathcal O(h)$ for a companion operator which maps the Adini rectangle to the Bogner--Fox--Schmit rectangle \cite{ciarlet1978finite} has been done in \cite{Brenner_ncfem}. For $\chi \in H_{\div}^B(\O)^d$ and $v_\disc \in X_{\disc,0}$, 
		 	cellwise integration by parts yields
		 	\be
		 	\begin{aligned}
		 		\int_\O \big( B\chi:\hbd v_\disc+ \mbox{div}(A\chi)\cdot\nabla_\disc v_\disc\big)\d\x=\sum_{\edge \in \edges}^{} 
		 		\int_{\edge}^{}(A\chi n_\edge)\cdot \llbracket \nabla_\disc v_\disc \rrbracket \d s(\x).\nonumber 
		 	\end{aligned}
		 	\ee
		 	From \cite[Theorem 7.2]{HDM_linear} and \eqref{def.WDtilde}, we deduce that ${\widetilde{W}}_\disc^B(\chi)=\mathcal O(h)$. Let $\phi \in H^4(\O)\cap H^2_0(\O)$. Introduce $\mbox{div}(A\chi)\cdot\nabla \phi$ in \eqref{def.WDtilde.nr}, use an integration by parts, the Cauchy--Schwarz inequality and Lemma \ref{lemma.interpolant} to obtain
		 	\begin{align}
		 	\big|\mathcal{\widetilde{W}}_\disc^B(\chi,\P_\disc\phi)\big|
		 	&\le\bigg|\int_\O \big( B\chi:\hbd \P_\disc \phi+\mbox{div}(A\chi)\cdot\nabla \phi\big)\d\x\bigg| \nonumber\\
		 	&\qquad \qquad+ \bigg|\int_{\O}\mbox{div}(A\chi)\cdot(\nabla_\disc \P_\disc \phi-\nabla \phi)\d\x\bigg|\nonumber\\
		 	=\bigg|\int_\O  B\chi:&(\hbd \P_\disc \phi-\hb \phi)\d\x\bigg|	+ \bigg|\int_\O \mbox{div}(A\chi)\cdot(\nabla_\disc \P_\disc \phi-\nabla \phi)\d\x\bigg|\le Ch^2.\nonumber
		 	\end{align}
		  The proof is complete by invoking Remark \ref{rates.PDE}, Lemma \ref{lemma.interpolant}, Lemma \ref{lemma.limitconformity} and Theorem \ref{th-h1-super}.
\smallskip

		 	$\bullet$ \textsc{Non-conforming FEM: the Morley triangle. }For the Morley element, there exists a companion operator such that $\omega(E_\disc)=\mathcal O(h)$, see \cite{Morley_plate} for more details. Let us estimate ${\widetilde{W}}_\disc^B(\chi)$, where $\chi \in H_{\div}^B(\O)^d$. For $v_\disc \in X_{\disc,0}$,
\be
		 	\begin{aligned}
		 		\int_\O \Big( B\chi:\hbd v_\disc+ \mbox{div}(A\chi)\cdot\nabla_\disc v_\disc\Big)=\sum_{\edge \in \edges}^{} 
		 		\int_{\edge}^{}(A\chi n_\edge)\cdot \llbracket \nabla_\disc v_\disc \rrbracket \d s(\x).\label{limit_conformity..}
		 	\end{aligned}
		 	\ee
		 	From \eqref{limit_conformity1} and \eqref{def.WDtilde}, we obtain $ {\widetilde{W}}_\disc^B(\chi)=\mathcal O(h)$. Let $\psi \in H^3(\O)\cap H^2_0(\O)$. In order to evaluate $\mathcal{\widetilde{W}}_\disc^B(\chi,\P_\disc\psi)$, introduce $\mbox{div}(A\chi)\cdot\nabla \psi$ in \eqref{def.WDtilde.nr}, use an integration by parts and the Morley interpolation property given by Lemma \ref{lemma.interpolant}. Hence, 
		 		\begin{align}
		 		\big|\mathcal{\widetilde{W}}_\disc^B(\chi,\P_\disc\psi)\big|
		 		&\le Ch^2+\bigg|\int_\O  B\chi:(\hbd \P_\disc \psi-\hb \psi)\d\x\bigg|.\nonumber
		 		\end{align}
		  Now, reproduce the same steps as in the limit conformity $\mathcal{W}_\disc^B(\xi,\P_\disc\psi)$ proof for Morley triangle (with $\xi=\chi$) and thus from \eqref{Morley.Pd1}--\eqref{Morley.Pd},
		 	$\mathcal{\widetilde{W}}_\disc^B(\chi,\P_\disc\psi)=\mathcal {O}(h^2).$ 
		 	\smallskip
		 	
		 As a consequence, for the Morley triangle, if $\bu \in H^4(\O)\cap H^2_0(\O)$, combine the above estimates, Theorem \ref{th:error.est.PDE}, Remark \ref{rates.PDE}, Lemmas \ref{lemma.interpolant}--\ref{lemma.limitconformity} and Theorem \ref{th-h1-super} to obtain the required result.
		 \end{proof}
		  
	\subsection{Technical results}
		 \begin{lemma}[Poincar\'e inequality along an edge in $L^2$ norm]\cite[Lemma A.1]{HDM_linear}\label{Poincare_edge_lp}
		 	Let $\edge$ be an edge of a polygonal cell, $w \in H^1(\edge)$ and assume that $w$ vanish at a point on the edge $\edge \in \edges$. Then 
		 	$\norm{w}{L^2(\edge)}\le h_{\edge} \norm{\nabla_{\mesh}w}{L^2(\edge)^d}$, where $h_\edge$ is the length of $\edge$.
		 \end{lemma}
		 
	\begin{lemma}\label{jump_bb_gradient}
		Let $k\ge0$ be an integer and $w \in \mathbb{P}_k(\mesh)$. If for all $\edge \in \edges$ there exists $x_{\edge} \in \edge$ such that $\llbracket w \rrbracket (x_{\edge})=0$, then there exists $C>0$ such that $\norm{w}{} \le C \norm{\nabla_{\mesh}w}{}.$
	\end{lemma}
	\begin{proof}
		Consider the $\norm{\cdot}{dG,h}$ norm defined by: For all $w \in H^1(\mesh)$,
		\be  \label{dg1}
		\norm{w}{dG,h}^2:=\norm{\nabla_\mesh w}{}^2 + \sum_{\edge \in \edges}^{}\frac{1}{h_\edge}\norm{\llbracket w \rrbracket}{L^2(\edge)}^2.
		\ee
	Since $\llbracket w \rrbracket (x_{\edge})=0$ for all $\edge \in \edges$, a use of Lemma \ref{Poincare_edge_lp} and the trace inequality (see \cite[Lemma 1.46]{DG_DA}) yields
		\begin{align}
		\norm{\llbracket w \rrbracket}{L^2(\edge)} & \le h_\edge \norm{\nabla_{\mesh} \llbracket w \rrbracket}{L^2(\edge)^d}\le h_\edge\sum_{\cell \in \mesh_\edge} \norm{\nabla_\mesh w_{\lvert \cell}}{L^2(\edge)^d}\nonumber\\
		& \le Ch_\edge\sum_{\cell \in \mesh_\edge}h_{\cell}^{-1/2} \norm{\nabla_\mesh w}{L^2(\cell)^d} \label{norm_jump1}
		\end{align}
		where $C>0$ depends only on $k$ and $\eta$.  A substitution of \eqref{norm_jump1} in \eqref{dg1} leads to
		\begin{align}
		\norm{w}{dG,h}^2&\le\norm{\nabla_\mesh w}{}^2 + 2\sum_{\edge \in \edges}^{} C_{}h_\edge\sum_{\cell \in \mesh_\edge}h_{\cell}^{-1} \norm{\nabla_\mesh w}{L^2(\cell)^{d}}^2\nonumber\\
		&\le \norm{\nabla_\mesh w}{}^2 + C \sum_{\cell \in \mesh}^{} \sum_{\edge \in \edgescv}^{} \norm{\nabla_\mesh w}{L^2(\cell)^{d}}^2\le C\norm{\nabla_\mesh w}{}^2. \nonumber
		\end{align}
	 Use the fact that $\norm{w}{} \le C \norm{w}{dG,h}$ (\cite[Theorem 5.3]{DG_DA}) to deduce $\norm{w}{} \le C\norm{\nabla_{\mesh}w}{}$. 	
	\end{proof}	
 \textbf{Acknowledgment: }The author would like to sincerely thank Prof J\'er\^ome Droniou and Prof Neela Nataraj for their fruitful comments. 
\bibliographystyle{abbrv}
\bibliography{Fourth_order_elliptic}

\end{document}